\documentclass[11pt]{amsart}
\usepackage{amssymb,latexsym,amsmath,amsthm,enumitem,mathrsfs}
\usepackage[margin=1in]{geometry}
\usepackage{hyperref}
\usepackage{fancyhdr}
\usepackage{color}
\usepackage{stmaryrd}
\usepackage{mathrsfs}
\usepackage{lineno}

\numberwithin{equation}{section}

\newtheorem{thm}{Theorem}[section]
\newtheorem*{thm*}{Theorem}
\newtheorem{prop}[thm]{Proposition}
\newtheorem{lemma}[thm]{Lemma}

\theoremstyle{remark}
\newtheorem{remark}[thm]{Remark}

\newcommand{\Span}{\mathrm{Span}}

\newcommand{\tu}{\tilde{u}}
\newcommand{\bF}{\mathbf{F}}

\newcommand{\bc}{\mathbf{C}}
\newcommand{\bb}{\mathbf{B}}
\DeclareMathOperator{\adj}{adj}

\newcommand{\C}{\mathbb{C}}

\newcommand{\R}{\mathbb{R}}
\newcommand{\T}{\mathbb{T}}
\newcommand{\Z}{\mathbb{Z}}

\newcommand{\ep}{\varepsilon}

\newcommand{\con}{\equiv}

\newcommand{\bstack}[2]{\substack{{#1}\\{#2}}}

\newcommand{\maps}{\rightarrow}

\newcommand{\al}{\alpha}

\newcommand{\ga}{\gamma}
\newcommand{\del}{\delta}
\newcommand{\Del}{\Delta}

\newcommand{\sig}{\sigma}
\newcommand{\lam}{\lambda}

\newcommand{\Pcal}{\mathcal{P}}

\newcommand{\Mcal}{\mathcal{M}}

\newcommand{\Qcal}{\mathcal{Q}}

\newcommand{\Tcal}{\mathcal{T}}

\newcommand{\Bbf}{\mathbf{B}}

 \newcommand{\beq}{\begin{equation}}
\newcommand{\eeq}{\end{equation}}

\newcommand{\w}{\omega}
\newcommand{\g}{\gamma}

\newcommand{\z}{\zeta}

\renewcommand{\a}{\alpha}

\begin{document}

\title[On Polynomial Carleson operators along quadratic hypersurfaces]{On Polynomial Carleson operators along quadratic hypersurfaces}

\author[Anderson, Maldague, Pierce and Yung]{Theresa C. Anderson, Dominique Maldague, Lillian B. Pierce, and Po-Lam Yung}
\address{Carnegie Mellon University
Wean Hall, Hammerschlag Dr., Pittsburgh, PA 15213 USA}
\email{tanders2@andrew.cmu.edu}
\address{Department of Mathematics, Massachusetts Institute of Technology, Cambridge, MA 02142-4307 USA}
\email{dmal@mit.edu}
\address{Department of Mathematics, Duke University, 120 Science Drive, Durham NC 27708 USA \& Hausdorff Center for Mathematics, Bonn, Germany}
\email{pierce@math.duke.edu}
\address{Mathematical Sciences Institute, Australian National University, Canberra, ACT 2601 \& Department of Mathematics, The Chinese University of Hong Kong, Shatin, Hong Kong}
\email{polam.yung@anu.edu.au, \, plyung@math.cuhk.edu.hk}


%
\begin{abstract}
We prove that a maximally modulated singular oscillatory integral operator along a hypersurface defined by $(y,Q(y))\subseteq \R^{n+1}$, for an arbitrary non-degenerate quadratic form $Q$, admits an \emph{a priori} bound on $L^p$    for all $1<p<\infty$, for each $n \geq 2$. This operator takes the form of a polynomial Carleson operator of Radon-type, in which the maximally modulated phases   lie in the real span of $\{p_2,\ldots,p_d\}$ for any set of fixed real-valued polynomials $p_j$ such that $p_j$ is homogeneous of degree $j$, and $p_2$ is not a multiple of $Q(y)$. The general method developed in this work applies to quadratic forms of arbitrary signature, while previous work considered only the special positive definite case $Q(y)=|y|^2$.
\end{abstract}

\maketitle

\begin{center}
\emph{Dedicated to Lennart Carleson, on the occasion of his 96th birthday}
\end{center}

\section{Introduction}
In this work, we study maximally modulated operators of Radon-type. Precisely, let integers $n \geq 1$ and $d \geq 2$ be fixed.  For each fixed real-valued polynomial $P$ on $\R^n$, define an operator, initially acting on Schwartz functions, by
\beq\label{RP_Q_dfn}
 R_P f(x) = \int_{\R^n} f(x-\ga(y)) e^{i P(y)} K(y) dy,
 \eeq
where $\ga(y) = (y,Q(y)) \subset \R^{n+1}$ is a hypersurface   defined by a fixed non-degenerate quadratic form $Q : \R^n \maps \R$  in $n$ variables, and $K$ is a Calder\'on-Zygmund kernel (see (\ref{CZ_dfn})). (Recall that a quadratic form $Q$ on $\R^n$ is said to be non-degenerate if the associated $n\times n$ symmetric matrix, also denoted $Q$, such that $Q(y) = y^t Q y$, has the property that $\det Q \neq 0$.) We study the   corresponding maximally modulated operator, defined by 
\beq\label{Carlhyper}
f(x) \mapsto \sup_{P \in \mathcal{P}} |R_P f(x)|,
\eeq
in which $\Pcal$ is a chosen set of polynomials.
It is reasonable to expect that the operator (\ref{Carlhyper}) satisfies an \emph{a priori} bound on $L^p(\R^n)$ for all $1<p<\infty$, for any class $\Pcal$ of polynomials of bounded degree. We prove this for arbitrary quadratic hypersurfaces $\ga(y) = (y,Q(y))$ for all $n \geq 2$, when $\Pcal$ is a class of polynomials defined as the real span of a set of nonlinear homogeneous polynomials, whose quadratic component is not $Q$.

\begin{thm}\label{thm_main_R}
Fix integers $n \geq 2$ and  $d \geq 2$. Let $\ga: \R^n \mapsto \R^{n+1}$ be   defined by 
\[ \{ (y,Q(y)) : y \in \R^n\}\]
for a non-degenerate quadratic form $Q \in \R[X_1,\ldots, X_n]$. Let $p_2(y),\ldots, p_d(y)$ be a fixed set of real-valued polynomials on $\R^n$ such that each $p_j(y)$ is homogeneous of degree $j$ and $p_2(y)  \neq C Q(y)$ for any $C \neq 0$. Let $\Qcal_d$ denote the class of real-valued polynomials 
\beq\label{Q_class}
\Qcal_d = \Span_\R\{p_2,p_3,\ldots,p_d\}.
\eeq
For each $P \in \Qcal_d$, define the Carleson operator $R_P$ as in (\ref{RP_Q_dfn}). For each $1<p<\infty$ the following \emph{a priori} inequality holds for all Schwartz class functions $f \in \mathcal{S}(\R^{n+1})$: 
\beq\label{R_sup_ineq}
 \| \sup_{P \in \Qcal_d} |R_P f| \|_{L^p(\R^{n+1})} \leq A \|f\|_{L^p(\R^{n+1})},
 \eeq
with a constant $A$ that may depend on $n,d,p, Q$ and the fixed polynomials $p_2, \ldots, p_d$.
\end{thm}
 A simple change of variables (see \S \ref{sec_prelim}) shows that it suffices to prove Theorem \ref{thm_main_R} in the case that 
\[
Q(y) = \theta_1 y_1^2 + \cdots + \theta_n y_n^2, \qquad \theta_i \in \{ +1,-1\}.
\]
If $r$ of the signs $\theta_i$ are positive,   the form is said to have signature $(r,n-r)$. The special case of Theorem \ref{thm_main_R} for signature $(n,0)$ with $Q(y)=|y|^2$, so that $\ga(y) = (y,|y|^2)$ defines a paraboloid, recovers the main theorem of Pierce and Yung \cite{PieYun19}, which introduced the study of Carleson operators of Radon-type. The new work of this paper enables us to treat Radon-type behavior defined by quadratic forms of \emph{arbitrary} signature.  To do so, we develop a new perspective that provides a  versatile framework for extracting decay from an oscillatory integral operator at the heart of the argument. This perspective replaces complicated ad hoc arguments in the earlier work of Pierce and Yung, and will be useful to study even broader classes of Carleson operators of Radon-type. We now describe the context for these operators, and then describe the present work.

 \subsection{Relation to previous work}
Carleson operators have their genesis in Carleson's 1966 work \cite{Car66} on the convergence of Fourier series. Carleson's theorem can now be summarized as the statement that the operator 
$f \mapsto \sup_{\lam \in \R} |T_\lam f(x) |$
satisfies an $L^2$ bound, where for each $\lam \in \R$ we define 
\[ T_\lam f (x) = \mathrm{p.v.} \int_{\T} f(x-y) e^{i \lam y} \frac{dy}{y}.\]
Carleson's work was followed closely by the proof of $L^p$ bounds for $1<p<\infty$ by Hunt \cite{Hun68}, an $\R^n$ version by Sj\"olin \cite{Sjo71}, and importantly, a new method of proof by C. Fefferman \cite{Fef73}. These approaches then led to other celebrated works such as that of Lacey and Thiele \cite{LacThi} on the bilinear Hilbert transform. 

Elias M. Stein expanded the field of vision by asking  whether $L^p$ bounds hold for all $1<p<\infty$ for polynomial Carleson operators. For each fixed real-valued   polynomial $P$ on $\R^n$, define 
\[ T_P f(x) = \int_{\R^n} f(x-y) e^{i P (y)} K(y) dy,\]
for $K$ a Calder\'on-Zygmund kernel (see (\ref{CZ_dfn})). The corresponding Carleson operator, over a class $\Pcal$ of polynomials, takes the form
\beq\label{Polycar}
f \mapsto \sup_{P \in \mathcal{P}} |T_P f(x)|.
\eeq
When $\Pcal$ is the class $\Pcal_d$ of all real-valued polynomials on $\R^n$ of degree at most $d$ (for a fixed $d$), this is called the polynomial Carleson operator. Stein initially proved (\ref{Polycar}) is bounded on $L^2(\R)$ for $\Pcal = \Span\{y^2\}$ \cite{Stein95}; this was based on special properties of Gaussians. Next, Stein and Wainger developed a $TT^*$ argument to prove (\ref{Polycar}) is bounded on $L^p(\R^n)$ for $1<p<\infty$ when $\Pcal = \Span \{ y^\al : \al \in \Z_{\geq 0}^n, 2 \leq |\al| \leq d\}$; that is, $\Pcal$ is comprised of all polynomials of degree at most $d$, with no linear terms \cite{SWCarl}. In dimension $n=1$, Lie  proved $L^p$ bounds for all $1<p<\infty$ for the full class $\Pcal_d$ with no restrictions \cite{Lie09,Lie20}. Finally, independent work of Lie and Zorin-Kranich resolved Stein's question in arbitrary dimensions  \cite{Lie17x,Zor21}; see also the comprehensive survey of Lie \cite{Lie24}.

The present manuscript fits in the broader context of a further question raised by Stein:   whether an $L^p$ bound holds for  a Carleson operator of Radon-type. Let $K$ again be a Calder\'on-Zygmund kernel and let $\ga : \R^n \maps \R^m$ define an $n$-dimensional submanifold in $\R^m$ of finite type (that is, such that it has at most a finite order of contact with any affine hyperplane).
Then define for each fixed real-valued polynomial $P$ on $\R^n$,  an operator initially acting on Schwartz functions by
\beq\label{RP_dfn}
 R_P^\ga f(x) = \int_{\R^n} f(x-\ga(y)) e^{i P(y)} K(y) dy.
 \eeq
The corresponding maximally modulated Carleson operator is then of the form 
\beq\label{RP_sup_dfn}
f \mapsto \sup_{P \in \mathcal{P}} |R_P^\ga f(x)|,
\eeq
in which $\Pcal$ is a chosen class of polynomials, such as $\Pcal_d$.
This setting was first studied by Pierce  and Yung \cite{PieYun19}, who in particular considered the paraboloid $\ga(y) = (y,|y|^2) \subset \R^{n+1}$ for $n \geq 2$. That work proved that (\ref{RP_sup_dfn}) satisfies an \emph{a priori} bound on $L^p(\R^{n+1})$ for all $1<p<\infty$ for the class $\Pcal = \Span\{p_2,p_3, \ldots, p_d\}$ where $p_2, p_3, \ldots, p_d$ are $d-1$ fixed polynomials with the property that $p_j$ is homogeneous of degree $j$, and $p_2\not\in \mathrm{Span}_\R\{ |y|^2\}$. This work generalized the $TT^*$ methods of Stein and Wainger \cite{SWCarl}, while also introducing techniques from   Littlewood-Paley theory,  square functions, and   delicate arguments with oscillatory integral estimates. The restriction that $p_2\not\in \mathrm{Span}_\R\{ |y|^2\}$ was a natural consequence of applying   $TT^*$ methods to a maximally modulated operator: such methods are only expected to succeed when the phase has sufficient independence from the defining function of the submanifold $\ga(y)$.

 \subsection{Overview of the argument}\label{sec_roadmap}
 In   this paper, we prove $L^p$ bounds for  hypersurfaces $\ga(y) =(y,Q(y))$ for quadratic forms $Q$ of arbitrary signature. Our method also has a $TT^*$ argument at its core, and a crucial part of the argument is to extract decay from an oscillatory integral operator.
We divide the proof of Theorem \ref{thm_main_R} into five steps. Steps 1-4 are natural generalizations of the method developed in \cite{PieYun19}, and although we provide a  self-contained treatment of these generalizations, we are efficient in our exposition. Step 5 departs significantly from the approach of \cite{PieYun19} and is the   novel technical heart of the paper.
 For the convenience of the reader, we now give a concise roadmap of these steps. As the main interest is in quadratic hypersurfaces  $\ga$ with signature $(r,n-r)$ with $r \geq 1$, not previously covered by the work of \cite{PieYun19}, for simplicity we will   refer to $\ga$ as a hyperboloid.

{\bf Step 1} (\S \ref{sec_red_L2}) reduces the proof of Theorem \ref{thm_main_R} to proving an $L^2$ estimate for an auxiliary operator. Fix a set of homogeneous polynomials $p_1,\ldots, p_M$, and let 
$P_\lam(y) = \sum_j \lam_j p_j(y)$, $\|\lam\| = \sum_j |\lam_j|$. 
Define for $\eta$ in an appropriate class of $C^1$ functions and real scaling parameter $a>0$, the operator
\beq\label{I_a_intro}
^{(\eta)}I_a^\lam f(x,t) = \int_{\R^n} f(x-y,t-Q(y)) e^{iP_\lam(y/a)} \frac{1}{a^n} \eta\left( \frac{y}{a} \right) dy, \qquad (x,t) \in \R^{n+1}.
\eeq
This integrates over a portion of the hyperboloid.
Roughly speaking, we show that Theorem \ref{thm_main_R} holds if there exists some $\del>0$ such that for every $r \geq 1$,
\beq\label{I_intro}
\| \sup_{r \leq \| \lam \|\leq 2r} \sup_{k \in \Z} |^{(\eta)}I_{2^k}^\lam f| \,\|_{L^2(\R^{n+1})} \ll r^{-\del} \|f\|_{L^2(\R^{n+1})}.\eeq
This reduction is encapsulated in Theorem \ref{thm_I_to_R}, and its proof uses results for maximally truncated singular Radon transforms over $\ga$, which we recall in \S \ref{sec_prelim}. 
Since it is an interesting open question to study a maximally modulated operator of the form (\ref{RP_sup_dfn}), over an appropriate  class $\Pcal$ of phases, for other types of  submanifolds $\ga$, we demonstrate this initial step   in a quite general setting.  

{\bf Step 2} (\S \ref{sec_square})   introduces a  maximal version of a ``flatter''  auxiliary  operator. Rather than integrating over a hyperboloid, this auxiliary operator integrates over $\R^{n+1}$ but with anisotropic scaling compatible with the hyperboloid. Define for each $C^1$ bump function $\eta$ and real scaling parameter $a>0$, the operator
\[
 ^{(\eta)}J_a^\lam f(x,t) = \int_{\R^{n+1}} f(x-y,t-s) e^{i P_\lam(y/a)} \frac{1}{a^n} \eta( \frac{y}{a}) \frac{1}{a^{2}} \zeta(\frac{s}{a^{2}}) dy ds, \qquad (x,t) \in \R^{n+1}.
\]
This   operator satisfies   an $L^2$ bound analogous to (\ref{I_intro}), by the work of \cite{SWCarl}.
The heart of the second step is to introduce a square function that encapsulates the comparison of the auxiliary operator (\ref{I_a_intro}) to this flatter analogue. Roughly speaking, the square function, say $S_r(f)$, takes the following form:
\beq\label{S_crude}
( \sum_{k \in \Z} ( \sup_{r \leq \|\lam\|< 2r} | ( ^{(\eta_k)}I_{2^k}^\lam - ^{(\eta_k)}J_{2^k}^\lam) f| )^2 )^{1/2}.
\eeq
(More precisely, it is defined using a Littlewood-Paley decomposition that localizes the Fourier support of $f$ in the last variable.)
Bounding this square function is then the key to proving (\ref{I_intro}). 
The central goal is to prove that there exists $\del>0$ such that for any Schwartz function $f$ on $\R^{n+1}$, and any $r \geq 1$,
\beq\label{S_r_intro}
\| S_r(f)\|_{L^2(\R^{n+1})} \ll r^{-\del} \|f\|_{L^2(\R^{n+1})}.
\eeq
In order to verify this, one can see from the crude definition of $S_r(f)$ suggested by (\ref{S_crude}) that it is natural to aim to prove that uniformly in $k \in \Z$,
\beq\label{IJ_rough}
\| \sup_{r \leq \|\lam\|< 2r} | ( ^{(\eta_k)}I_{2^k}^\lam - ^{(\eta_k)}J_{2^k}^\lam) f|\, \|_{L^2(\R^{n+1})}\ll r^{-\del} 2^{-\ep |k|}\|f\|_{L^2(\R^{n+1})},
\eeq
for some $\del>0,\ep>0$.
(More precisely, one must also prove an estimate that takes into account the Littlewood-Paley decomposition.)
Theorem \ref{thm_S_to_IJ} formally records that if an appropriate bound roughly of the form (\ref{IJ_rough}) holds, then the desired square function estimate (\ref{S_r_intro}) holds on $L^2$. Thus in order to complete the proof of our main theorem, we must verify that the hypothesis (\ref{IJ_rough}) of Theorem \ref{thm_S_to_IJ} is true.  (Like Step 1, Step 2 also works for more general manifolds $\ga$.)

{\bf Step 3} (\S \ref{sec_change_var}) defines a crucial change of variables. Roughly speaking, fix $k \in \Z$ and let $T$ denote the operator 
\[
f \mapsto \sup_{r \leq \|\lam\|< 2r} | ( ^{(\eta_k)}I_{2^k}^\lam - ^{(\eta_k)}J_{2^k}^\lam) f|.
\]
(More precisely, to define $T$ we linearize this operator, using stopping-times.)
Our general strategy for proving (\ref{IJ_rough}) is to bound $TT^*$ on $L^2$, with a norm that exhibits decay like $r^{-\del}2^{-\ep |k|}$ for some $\del,\ep>0$. In Step 3, we define a change of variables that will allow us to show that $TT^* f(x,t)$ can be written as a sum of $n$ convolutions $(f * K_{\sharp,l}^{\nu,\mu})(x,t)$ for an appropriate kernel $K_{\sharp,l}^{\nu,\mu}$, for each $1 \leq l \leq n$. Roughly speaking, if $TT^*f$ is initially an integral over $(u,z) \in \R^n \times \R^n$, for each $1 \leq l \leq n$ we construct a change of variables to $(u,\tau,\sig) \in \R^n \times \R \times \R^{n-1}$ so that  
\[ K_{\sharp,l}^{\nu,\mu}(u,\tau)=\int_{\R^{n-1}}e^{iP_\nu(u+z)-iP_\mu(z)}\Psi(u,z)d\sigma, \qquad (u,\tau) \in \R^{n+1},
\]
where $\sig \in \R^{n-1}$ is defined implicitly in terms of $u,z$ for every $z \in \R^n$, $\Psi(u,z)$ is a $C^1$ bump function, and $\nu,\mu$ are arbitrary stopping-times. We do this for each $ 1 \leq l \leq n$, depending on whether $u$ lies in the $l$-th sector of a partition of $\R^n$; this guarantees the compact support of $\sig$ in the region of integration.
To define the change of variables in Step 3, we crucially use the fact that in our main theorem we assume that $\ga(y) =(y,Q(y))$ for a non-degenerate \emph{quadratic} form $Q$.

{\bf Step 4} (\S \ref{sec_TT_to_K}) shows that if an appropriate upper bound holds pointwise for the kernel $K_{\sharp,l}^{\nu,\mu}(u,\tau)$ defined in Step 3, then the hypothesis (\ref{IJ_rough}) of Theorem \ref{thm_S_to_IJ} is true. This deduction is encapsulated in Theorem \ref{thm_K_to_S}.
This is the step in which $TT^*$ methods are clearly applied.

{\bf Step 5} (\S \ref{sec_K}) proves that the desired pointwise upper bound for $K_{\sharp,l}^{\nu,\mu}(u,\tau)$ holds, as stated in Theorem \ref{thm_main_K}. This step crucially uses the assumption (\ref{Q_class}) on the class of polynomials $\Qcal_d$ over which the maximal modulation occurs in Theorem \ref{thm_main_R}. 
We use the explicit change of variables chosen in  Step 3 in order to express $K_{\sharp,l}^{\nu,\mu}(u,\tau)$ as an oscillatory integral with a phase that depends on stopping-time functions. The key is to show that this phase is sufficiently large (almost all the time), so that the oscillatory integral has a satisfactory upper bound (via a van der Corput estimate). To accomplish this, we develop a new framework which differs conceptually from \cite{PieYun19}, and is both simpler and more flexible.

 An essential difficulty is that the $TT^*$ argument introduces two different stopping-time functions, one of which depends on variables of integration (and is hence ``bad'') and one of which does not (and is hence ``good''). It is critical that the pointwise upper bound for $K_{\sharp,l}^{\nu,\mu}(u,\tau)$ has no dependence on the bad stopping-time function, yet the phase in this oscillatory integral includes terms that depend on both the good and the bad stopping-times. We package the portion of the phase that depends on the bad stopping-time as the image of a linear operator; then by projecting onto a subspace orthogonal to this image, we can study an expression from which the bad stopping-time has been removed. We show that this expression is almost always large,  and then deduce that the phase of $K_{\sharp,l}^{\nu,\mu}(u,\tau)$ is almost always large, allowing us to complete the argument.

The study of Carleson operators of Radon-type, of the form (\ref{RP_sup_dfn}) for a given submanifold $\ga$, has only recently begun, after they were introduced in \cite{PieYun19}.  Several intrinsic difficulties of bounding such operators have been described in \cite[\S 2]{PieYun19}; here, we remark specifically on the restrictions present in Theorem \ref{thm_main_R}.
First, the case of dimension $n=1$ is out of reach of the methods of this paper, effectively due to dimension counting; see Remark \ref{remark_ngeq2}. For $n=1$, it is natural to ask whether for a monomial curve $\ga(y) = (y,y^m)$, the operator (\ref{RP_sup_dfn}) is bounded on $L^p$ for $1<p<\infty$ for $\Pcal = \Pcal_d$ the class of polynomials of degree at most $d$. For $m=1$ this can be reduced to an instance of the original Carleson theorem; see \cite[p. 2980]{GPRY17}. This question is open for degree $m \geq 2$; the first results in this direction, weaker than Theorem \ref{thm_main_R}, were   established by   Guo, Pierce, Roos and Yung \cite{GPRY17}. Recently, other authors have obtained interesting results on a cluster of    closely related questions (Carleson operators with anisotropic scalings or with fewnomial phases), e.g. \cite{Guo17,Roo19,Ram21a,Ram21b,Bec22x}.   Nevertheless,   fundamental questions about operators of the form (\ref{RP_sup_dfn}) remain open for $n=1$. 
 When $\ga(y) = (y,Q(y))$ is defined by a non-degenerate quadratic form and $n \geq 2$, our methods impose that  $\Span\{p_2,\ldots,p_d\}$   contains no linear polynomials,  and that the (homogeneous) quadratic contribution is not in the span of $Q(y)$ (see Remarks \ref{remark_no_linear} and  \ref{remark_no_quad}).  These restrictions are a natural consequence of applying $TT^*$ methods; one might expect that time-frequency methods must be applied in order to avoid such  restrictions.  
More generally, it is interesting to ask whether the class $\Qcal_d$ in Theorem \ref{thm_main_R} could be allowed to have more degrees of freedom, and it would be very interesting to obtain results for (\ref{RP_sup_dfn}) for more general submanifolds $\ga$.

\section{Preliminaries}\label{sec_prelim}
In this section, we briefly record well-known properties of singular and maximal Radon transforms, which we will call upon in \S \ref{sec_red_L2}.
\subsection{Calder\'on-Zygmund kernels}
We work with
Calder\'on-Zygmund kernels $K$ defined as follows: $K$  is a tempered distribution that agrees with a $C^1$ function $K(x)$ for $x \neq 0$, such that 
\beq\label{CZ_dfn}
 |\partial_x^\al K(x)| \ll |x|^{-n-|\al|}, \qquad 0 \leq |\al| \leq 1,
 \eeq
 and $\hat{K}$ is an $L^\infty$ function.
Such a kernel admits a decomposition 
\[ K(x)=\sum_{j=-\infty}^\infty 2^{-nj}\phi_j(2^{-j}x)=:\sum_{j=-\infty}^\infty K_j(x)\]
    where each $\phi_j$ has the following properties (see  \cite[\S 5]{SWCarl}): 
    \begin{enumerate}[label=(\roman*)]
        \item $\phi_j$ is a $C^1$ function with support in $1/4<|x|\le1$,
        \item $|\partial_x^\a\phi_j(x)|\le C$ for $0\le|\a|\le 1$ for some constant $C$ that is uniform in $j$,
        \item $\int_{\R^n}\phi_j(x)dx=0$ for every $j$. 
    \end{enumerate}
    
 \subsection{Singular and maximal Radon transforms}
    Let $\ga  \subset \R^{m}$ be a submanifold described by a polynomial mapping  $\ga  = (\ga_1,\ldots, \ga_{m}) : \R^n \maps \R^m$ 
 where for each $i$, $\ga_i: \R^n \maps \R$ is a polynomial with real coefficients. The singular Radon transform
 \beq\label{sing_Rad_dfn}
 \Tcal f(x) = \int_{\R^n} f(x-\ga(y)) K(y) dy, \qquad x \in \R^m,
 \eeq
 initially defined for $f$ of Schwartz class, extends to a bounded operator on $L^p(\R^m)$ for $1<p<\infty$. This follows from \cite[Ch. XI \S 4.4-4.5]{SteinHA}. (Precisely, the argument presented in \S 4.5 of that text assumes that $K$ is homogeneous, but substituting the decomposition $K_j$ described above for the functions $K_j$ in that text, and using the properties (i), (ii), (iii) of the functions $\phi_j$ confirms that the resulting measures $\{dm^j\}$ satisfy the formalism in \S 4.4, and the argument can proceed verbatim.)

Next,  define the maximally truncated singular Radon transform
\beq\label{max_trunc_Rad_dfn} \sup_{t>0}|\Tcal_tf(x) | =  \sup_{t>0}| \int_{|y| \geq t} f(x-\ga(y)) K(y) dy|.
\eeq
When $\ga(y) = (y,Q(y))$ with $Q(y) = |y|^2$, \cite[Appendix]{PieYun19} proves that this operator extends to a bounded operator on $L^p(\R^{n+1})$ for all $1<p<\infty$; the same argument applies (with only superficial changes) for any diagonal form $Q(y) = \sum_{1 \leq i \leq n} \theta_i y_i^2$ for any signs $\theta_i \in \{\pm 1\}$; this is the only case we require to prove Theorem \ref{thm_main_R}. But in \S \ref{sec_red_L2} of our work, it is no additional trouble to consider more general submanifolds defined by any polynomial mapping $\ga:\R^n \maps \R^m$. Thus for such $\ga$, we note that the operator (\ref{max_trunc_Rad_dfn}) extends to a bounded operator on $L^p(\R^m)$ for all $1<p<\infty$ by   \cite[Thm. 1.30]{MSZ20_Boot}  (although with slightly different constraints on $K$, see Remark \ref{remark_jump}); see also \cite{DuoRub86,JSW08} for related results.

  Finally, define 
for a Schwartz function $f$,
\beq\label{M_rad_gamma}
\Mcal_\ga f(x) = \sup_{a>0} \int_{\R^n} |f (x-\ga(y))| \frac{1}{a^n} \chi_{B_1}(\frac{y}{a}) dy.
\eeq
For $\ga$ an $n$-dimensional submanifold in $\R^m$ of finite type (this is satisfied if $\ga$ is defined by a polynomial mapping),   $\Mcal_\ga$ satisfies an \emph{a priori} bound on $L^p$ for all $1 < p \leq \infty$ by \cite[Ch. XI Thm. 1]{SteinHA}.
  
  \subsection{Reduction to the simplest type of quadratic form}
In our main theorem we assume
$\ga(y) = (y,Q(y))$   for a non-degenerate quadratic form $Q$. We can restrict to proving Theorem \ref{thm_main_R} for 
\beq\label{Q_dfn_0}
Q(y)= \sum_{i=1}^n\theta_iy_i^2,\qquad\theta_i\in\{\pm 1\}.
\eeq
Indeed, suppose first of all that $Q(y)$ is a non-degenerate quadratic form in $\R[X_1,\ldots, X_n]$, with corresponding real symmetric $n \times n$ matrix $Q$ such that $Q(y) = y^T Qy$. Then by the spectral theorem, there exists an orthogonal matrix $B$ such that $BQB^{-1}=A=\mathrm{diag}(a_1,\ldots, a_n)$ with all $a_i \neq 0$.
Then for $P_\lam(y) = \sum_{1 \leq j \leq M} \lam_j p_j(y)$, our operator of interest is
 \begin{align*} R^Q_{P_\lam,K}f(x,t) &:= \int_{\R^n}f(x-y,t-Q(y)) e^{iP_\lam(y)}K(y)dy\\
 &= \int_{\R^n}f(x-y,t-(By)^TA(By)) e^{iP_\lam(y)}K(y)dy.
 \end{align*}
Define $\del_B f(x,t) = f(Bx,t)$; 
 $P_\lam^B (y) = \sum_{1 \leq j \leq M} \lam_j p_j^B(y)$ where $p_j^B(y) = p_j(B^{-1}y)$; and $K^B(y) = K(B^{-1}y)$, which is also a Calder\'on-Zygmund kernel as in (\ref{CZ_dfn}). Note that $p_2 \not\in \mathrm{Span}(Q)$ if and only if $p_2^B \not\in \mathrm{Span}(A)$. 
 After a change of variables $u=By$, we see that as an operator
 \[R^Q_{P_\lam,K}f = \del_B \circ R^A_{P^B_{\lam},K^B}\circ \del_{B^{-1}} f.\]  
Thus, to prove that  for any fixed Calder\'on-Zygmund kernel $K$,   and any fixed polynomials $p_1,\ldots, p_M$, the operator $f \mapsto \sup_{P \in \Span\{p_1,\ldots,p_M\}} |R_{P,K}^Q f|$ is bounded on $L^p$,   it suffices to prove that for any fixed Calder\'on-Zygmund kernel $K$,   and any fixed polynomials $p_1,\ldots, p_M$ the operator 
\[f \mapsto \sup_{P \in \Span\{p_1,\ldots,p_M\}} |R_{P,K}^A f|\]
is bounded on $L^p$. Similarly, another change of variables shows that we may rescale to the case where each $a_i = \pm 1$, so that from now on we may assume that $Q(y)$ takes the form (\ref{Q_dfn_0}). This specific form of $Q$ only becomes useful in Step 3 (\S \ref{sec_change_var}), when it motivates an explicit change of variables.

\section{Step 1: Reduction to an $L^2$ estimate}\label{sec_red_L2}
We begin by introducing an auxiliary operator, and show that an $L^2$ bound for this operator implies $L^p$ bounds for the Carleson operator (\ref{RP_sup_dfn})  for $1<p<\infty$. For our main theorem, we only apply this in the case that $\ga(y) =(y,Q(y))\subset \R^{n+1}$, but it is no trouble to work more generally in this section, for readers interested in the case where   $\ga : \R^n \maps \R^m$ is an $n$-dimensional submanifold in $\R^m$ defined by polynomials.

We now fix the class of polynomial phases we will consider.
Let $p_1,\ldots,p_M$ be fixed homogeneous polynomials in $\R[X_1,\ldots,X_n]$. For each $1 \leq m \leq M$, we will set $d_m$ to be the degree of $p_m$.
We  assume that the polynomials $p_m$ are linearly independent over $\R$, which causes no limitations in our ultimate results, which involve a supremum over polynomials in the span of $p_1,\ldots, p_M$. Precisely, we set 
\beq\label{P_span}
\Pcal = \Span_\R \{p_1,\ldots, p_M\} 
\eeq
  so that any $P \in \Pcal$ can be expressed as 
\[ P_\lam(y) = \sum_{m=1}^M \lam_m p_m(y),\]
for a certain $\lam = (\lam_1,\ldots, \lam_M) \in \R^M$. We do not yet need to make any further assumptions about the polynomials $p_m$.

Next we define an auxiliary operator. Fix a $C^1$ bump function $\eta$ supported in the unit ball $B_1$ in  $\R^n$ and a real number $a>0$. Then define the operator
\beq\label{I_dfn}
^{(\eta)}I_a^\lam f(x,t) = \int_{\R^n} f((x,t)-\ga(y)) e^{iP_\lam(y/a)} \frac{1}{a^n} \eta\left( \frac{y}{a} \right) dy.
\eeq
This is dominated pointwise by the maximal Radon transform $\Mcal_\ga$ defined in (\ref{M_rad_gamma}), so it admits an \emph{a priori} estimate  for $1<p\leq \infty$. Indeed,   if the bump function $\eta$ varies over a family $\{\eta_k\}_k$ of $C^1$ bump functions supported in $B_1$, with uniformly bounded $C^1$ norm,  then
\beq\label{I_trivial_p}
\|\sup_{\lam \in \R^M} \sup_{ k \in \Z}  |{^{(\eta_k)}I_{2^k}^\lam} f|\;\|_{L^p(\R^{n+1})}
\ll \|f\|_{L^p(\R^{n+1})}, \qquad 1 < p \leq \infty.
\eeq
This will serve as a trivial bound.
However,   it is possible to use oscillation within $^{(\eta_k)}I_{2^k}^\lam$ to prove that the $L^2$ norm exhibits decay, measured in terms of the ``size'' of $\lam$. For this purpose, 
we define the isotropic norm for $\lam \in \R^M$:
\beq\label{lam_norm_iso}
\|\lam\| = \sum_{m=1}^M |\lam_m|.
\eeq

The main result of this section is:
\begin{thm}\label{thm_I_to_R}
 Let $p_1(y),\ldots, p_M(y)$ with $y \in \R^n$ be a fixed set of linearly independent, homogeneous polynomials in $\R[X_1,\ldots,X_n]$, and for any $\lam \in \R^M$ set $P_\lam = \sum_{m=1}^M \lam_m p_m$. 
 Let $\ga = (y,Q(y)) \subset \R^{n+1}$.
 Suppose that for any family $\{\eta_k \}_{k \in \Z}$ of $C^1$ bump functions supported in the unit ball $B_1(\R^n)$ with $C^1$ norm uniformly bounded by $1$, there exists a fixed $\del>0$ such that for any Schwartz function $f$ on $\R^{n+1}$ and any $r \geq 1$, 
\beq\label{IP_ineq}
\| \sup_{\bstack{\lam \in \R^M}{r \leq \| \lambda \| < 2r}} \sup_{k \in \Z} | ^{(\eta_k)}I_{2^k}^\lam f(x,t)| \|_{L^2(\R^{n+1})} \ll r^{-\del} \|f\|_{L^2(\R^{n+1})},
\eeq
in which the implicit constant may depend only on $n$ and the fixed polynomials $p_1,\ldots,p_M$.
Then   for each $1<p<\infty$, we have the \emph{a priori} estimate
\beq\label{RP_ineq}
  \| \sup_{\lam \in \R^M} |R_{P_\lam}^\ga f| \|_{L^p(\R^{n+1})} \ll \|f\|_{L^p(\R^{n+1})},
\eeq
in which the implicit constant may depend only on $p, n,$ and the fixed polynomials $p_1,\ldots,p_M$.
\end{thm}
Note that in the notation of Theorem \ref{thm_main_R}, the left-hand side of (\ref{RP_ineq}) is $\| \sup_{P \in \Pcal} |R_{P}^\ga f| \|_{L^p(\R^{n+1})}$ for the class $\Pcal$  in (\ref{P_span}). (For simplicity we have stated this theorem for values $r \geq 1$, but of course the same method of proof applies for $r \geq c$ for any fixed universal constant $c$, as we apply in (\ref{c_below}).)
The theorem also holds for any submanifold defined by a polynomial mapping $\ga : \R^n \maps \R^m$ under appropriate assumptions on $K$ (summarized in Remark \ref{remark_jump}), by the same argument we now describe.

The proof closely follows the original ideas of Stein and Wainger \cite{SWCarl} as adapted by Pierce and Yung for the paraboloid \cite[\S 5]{PieYun19}; we generalize those steps briefly here. 
We recall that the operator $R^\ga_{P_\lam}$ is defined in (\ref{RP_dfn}), and we apply the dyadic decomposition $K=\sum_j K_j$ from (\ref{CZ_dfn}).
Given a fixed $\lam \in \R^M$, we divide the indices $j$ according to  whether the phase polynomial $P_\lambda(y)$ is small or large on the support of $K_j(y)$. For this purpose, it is useful to let $d_m$ denote $\deg p_m$ for each $1 \leq m \leq M$, and   to define the nonisotropic norm
    \[ N(\lambda)=\sum_{m=1}^M|\lambda_m|^{1/d_m}. \]
Note that there exists a   constant $c_0>0$, depending only on $M$ and the multi-set of degrees
$\{d_1,\ldots, d_M\}$, such that 
\beq\label{N_comp}
N(\lam) \leq c_0\|\lam\| \qquad \text{for $N(\lam) \geq 1$}.
\eeq
Additionally, for any $1 \leq m \leq M$, $|\lam_m| \leq N(\lam)^{d_m}$.

For each fixed $\lam \in \R^M$, we then set
    \[ K_\lambda^+= \sum_{2^j\ge 1/N(\lambda)}K_j, \qquad K_\lambda^-=\sum_{2^j<1/N(\lambda)}K_j. \]
 For each $\lam \in \R^M,$ we split $R^\ga_{P_\lam}f=R^{\ga,+}_{P_\lambda}f+R^{\ga,-}_{P_\lambda}f$, where
    \[ R_{P_\lambda}^{\ga,\pm}f(x,t)=\int_{\R^n}f((x,t)-\ga(y))e^{iP_\lam(y)}K_\lambda^{\pm}(y)dy. \]  

\subsection{The case of large oscillation}
We first bound the operator $\sup_\lam|R_{P_\lam}^{\ga,+}f(x,t)|$, which corresponds to the case in which  the phase $P_\lam(y)$ is large, and oscillation plays a role.    Define the nonisotropic scaling 
    \[ 2^j\circ\lambda=(2^{jd_m}\lambda_m)_m. \]
 Here we use the assumption that each polynomial $p_m$ is homogeneous of degree $d_m$, so that $P_{2^j\circ\lambda}(y/2^j)=P_\lambda(y)$. Moreover, the nonisotropic norm is homogeneous with respect to this scaling: $N(2^j\circ\lambda)=2^jN(\lambda)$. 
    Consequently, for each fixed $\lam$, the operator decomposes as
    \[ R_{P_\lam}^{\ga,+} =   \sum_{N(2^j \circ \lam) \geq 1} {^{(\phi_j)}I_{2^j}^{2^j\circ\lambda}}  . \]
Thus, for a Schwartz function $f$, we have the pointwise inequality
\begin{align} 
\sup_\lambda |R_{P_\lam}^{\ga,+}f(x,t)|&\le \sup_\lambda\sum_{N(2^j\circ\lambda)\ge  1}\sup_{k\in\Z}|^{(\phi_k)}I_{2^k}^{2^j\circ\lambda}f(x,t)| \nonumber \\
&\le \sum_{l=0}^\infty \sup_{2^l \le N(\lambda')< 2^{l+1} }\sup_{k \in \Z}|^{(\phi_k)}I_{2^k}^{\lambda'}f(x,t)|.\label{R_I_sum}
\end{align}
The hypothesis of Theorem \ref{thm_I_to_R} provides an $L^2$ estimate for each operator on the right-hand-side.
Precisely, by (\ref{N_comp}), for $r \geq 1$, $\{\lam : N(\lam) \geq  r  \} \subseteq \{ \lam : c_0\|\lam\| \geq r \}$. Thus for any $r \geq 1$,
\[
\sup_{ N(\lambda') \geq   r}\sup_{k \in \Z}|^{(\phi_k)}I_{2^k}^{\lambda'}f(x,t)|
\leq \sup_{ \|\lam'\| \geq r/c_0}\sup_{k \in \Z}|^{(\phi_k)}I_{2^k}^{\lambda'}f(x,t)|
\leq \sum_{2^v \geq r/c_0}  \sup_{2^v \leq  \|\lam'\| <2^{v+1}}\sup_{k \in \Z}|^{(\phi_k)}I_{2^k}^{\lambda'}f(x,t)|.
\]
By the key hypothesis (\ref{IP_ineq}),
for any $r \geq 1$, the $L^2$ norm thus satisfies
\beq\label{c_below}
\| \sup_{ N(\lambda') \geq r}\sup_{k \in \Z}|^{(\phi_k)}I_{2^k}^{\lambda'}f| \; \|_{L^2} \leq  \sum_{2^v \geq r/c_0} (2^v)^{-\del} \|f\|_{L^2} \ll r^{-\del} \|f\|_{L^2}.
\eeq
On the other hand, recall the trivial bound (\ref{I_trivial_p}), valid for all $1<p< \infty$. Interpolation of (\ref{c_below}) and (\ref{I_trivial_p})
shows that for each $1<p<\infty$, there exists $\del(p)$ such that for any $r \geq 1$,
\beq\label{interpolation_I}
\| \sup_{ N(\lambda') \geq  r}\sup_{k \in \Z}|^{(\phi_k)}I_{2^k}^{\lambda'}f| \; \|_{L^p} \ll r^{-\del(p)} \|f\|_{L^p}.
\eeq
We remark that this interpolation argument leading to (\ref{interpolation_I}) requires several standard intermediate steps in order to pass from the \emph{a priori} inequalities to appropriate versions for all simple functions, which can then be interpolated. We refer to the detailed treatments of these intermediate steps for the exemplary case of maximal operators over curves, in \cite[Part II \S5]{SWCurv}.
Applying (\ref{interpolation_I}) in (\ref{R_I_sum}) after taking norms, shows that for each $1<p<\infty$,
\[\|\sup_\lambda |R_{P_\lam}^{\ga,+}f(x,t)|\|_{L^p} \leq \sum_{l=0}^\infty 2^{-l \del(p)} \|f\|_{L^p}\ll \|f\|_{L^p}.\]

\subsection{The case of small oscillation}
Next we consider the portion $R_{P_\lam}^{\ga,-}$ of the operator $R_{P_\lam}^\ga$, for which $P_\lam(y)$ is small, and oscillation does not play an important role.   The support of the kernel $K_\lambda^-(y)$ is contained in $|y|\le 1/N(\lambda)$. On this set, 
    \[ |e^{iP_\lam(y)}-1|\ll |P_\lam(y)|\leq \sum_{m=1}^M |\lambda_m||p_m(y)|\ll\sum_{m=1}^M (N(\lambda)|y|)^{d_m} \ll N(\lambda)|y|, \]
    where in the last inequality we applied the fact that $N(\lambda)|y|  \leq 1$. 
    The implicit constants   depend only the set of fixed polynomials $\{p_1,\ldots,p_M\}$. 
    For fixed $\lam$, we can then write
 \begin{multline}\label{R_minus}
 R_{P_\lam}^{\ga,-}f(x,t)=\int_{|y| \leq 1/N(\lam)}f((x,t)-\ga(y))K_\lambda^-(y)dy \\+O\left(N(\lambda)\int_{|y|\le 1/N(\lambda)}|y|^{-n+1}|f((x,t)-\g(y))|dy\right). \end{multline}
    The first term is controlled by a maximally truncated singular Radon transform, which is bounded on $L^p(\R^{n+1})$  for all $1<p<\infty$ by the observations recorded for (\ref{max_trunc_Rad_dfn}). 
The second term is controlled by the maximal operator $\Mcal_\ga$ defined in (\ref{M_rad_gamma}), since  
\begin{align*}
  N(\lam)\sum_{2^l \leq N(\lam)^{-1}} 2^{-l(n-1)} & \int_{2^{l-1} \leq |y| < 2^l} |f((x,t)-\ga(y))|   dy\\
&\ll  N(\lam)\sum_{2^l \leq N(\lam)^{-1}} 2^l \frac{1}{|B_{2^l}|} \int  |f((x,t)-\ga(y))| \chi_{B_{2^l}}(y)  dy
\\
& \ll N(\lam) N(\lam)^{-1} \Mcal_\ga f(x,t).
\end{align*}
In conclusion, for each $1<p<\infty$, 
$\|\sup_\lambda |R_{P_\lam}^{\ga,-}f|\;\|_{L^p}  \ll \|f\|_{L^p}.$
This completes the proof of Theorem \ref{thm_I_to_R}.

\begin{remark}\label{remark_jump} 
For readers interested in the more general case when $\ga$ is a polynomial mapping, we summarize how the fact that the operator (\ref{max_trunc_Rad_dfn})  extends to a bounded operator on $L^p(\R^m)$ for $1<p<\infty$ is deduced from   \cite{MSZ20_Boot}. Their work applies to kernels $K$ with the properties: (a) (boundedness) $|K(y)|\le C|y|^{-n}$ for all nonzero $y\in\R^n\setminus\{0\}$;  (b) (cancellation) $\int_{r \leq |y| \leq R} K(y) dy =0$ for $0<r<R< \infty$;  (c) (smoothness)
$|K(x) - K(x+y)| \leq \omega_K(|y|/|x|) |x|^{-n} $
for all $x,y$ such that $|y| \leq |x|/2$; here $\omega_K$ is a modulus of continuity. 
For such a kernel, \cite[Thm. 1.30]{MSZ20_Boot} states that if $\| \omega_K^\theta\|_{\mathrm{logDini}} + \| \omega_K^{\theta/2}\|_{\mathrm{Dini}} < \infty $ for some $\theta \in (0,1]$ then for every $p \in \{ 1+\theta, (1+\theta)'\}$ and $f \in L^p(\R^m)$, 
$ J_2^p( (\Tcal_t f)_{t>0}: \R^m \maps \C) \ll_{m,p} \|f\|_{L^p}.$
Here the jump norm is
 \begin{eqnarray*}
 J_2^p ( (\Tcal_t f)_{t>0} : \R^m \maps \C)
 	& = &\sup_{\lam>0} \| \lam (N_\lam (\Tcal_t(f): t >0))^{1/2} \|_{L^p}\\
	 & = & \sup_{\lam>0} \| \lam (\sup\{ J : \min_{0 <j \leq J} | \Tcal_{t_j} f(\cdot) - \Tcal_{t_{j-1}}f(\cdot)| \geq \lambda\})^{1/2} \|_{L^p}
	 \end{eqnarray*}
	 in which the minimum is over all $0< j \leq J$ and all $t_0< \cdots < t_J$ with $t_j >0$. In particular, by \cite[Lemma 2.3]{MSZ20_Interp}, for a given $1<p<\infty$ for which Theorem 1.30 in \cite{MSZ20_Boot} holds, then for $r \in (2,\infty]$ the $r$-variational semi-norm $V^r(\Tcal_t f: t>0)$ satisfies 
 \beq\label{MSZK_result} \|V^r(\Tcal_t f : t>0)\|_{L^{p,\infty}} \ll_{p,r} J_2^p(\Tcal_t f : t>0) \ll_{m,p} \|f\|_{L^p}.\eeq

If for example $K$ is homogeneous of degree $-n$, smooth away from the origin, and satisfies the cancellation condition $\int_{|y|=1}K(y)d\sig(y)=0$, then it certainly satisfies the above properties (a), (b), (c) with   $\omega_K(t)=ct$ for an appropriate constant $c\in(0,\infty)$. 
Similarly if $K$ is $C^1$ away from the origin, satisfies (\ref{CZ_dfn}) and (b) (which is somewhat stronger than the assumption $\hat{K} \in L^\infty$) then again all three properties are satisfied with  $\omega_K(t)=ct$ for an appropriate constant $c\in(0,\infty)$. 
For $w_K(t)=ct$,  the Dini and $\log$-Dini norms (see \cite[Eqn. (1.24)]{MSZ20_Boot}) of $(\w_K(t))^\theta=c^\theta t^\theta$ are finite for every $\theta\in (0,1]$. Therefore, by (\ref{MSZK_result}) we conclude that for each $p\in(1,\infty)$ and $r\in (2,\infty]$ 
\beq\label{V_term} \|V^r(\Tcal_t f:t>0)\|_{L^{p,\infty}}\ll_{p,r}\|f\|_{L^p}.  \eeq
For $r=\infty$ we have the pointwise inequality
\[  \sup_{t>0}|\Tcal_t f(x)|\le |\Tcal f(x)|+V^r(\Tcal_t f(x):t>0). \]
The operator $\Tcal f(x)$ extends to a bounded operator on $L^p(\R^m)$  for $1<p<\infty$ (recall (\ref{sing_Rad_dfn})), and we apply (\ref{V_term}) to the second term on the right-hand side.
Finally, the strong $L^p$ boundedness of $f \mapsto \sup_{t>0}|\Tcal_t f|$ follows after invoking the Marcinkiewicz interpolation theorem with any exponents $p_1,p_2$ satisfying $1<p_1<p<p_2<\infty$.  
 
\end{remark}

\section{Step 2: Passage to a square function}\label{sec_square}
In this section, we show that the key hypothesis (\ref{IP_ineq}) of Theorem \ref{thm_I_to_R}  can be verified if a related square function has an $L^2$ bound with decay in the norm $\|\lam\|$ defined in (\ref{lam_norm_iso}). We only require the case that $\ga(y) = (y,Q(y))$, but the reader will note that the argument in this section may be adapted to the case where
 $\ga(y)$ is a hypersurface in $\R^{n+1}$, of the form 
\beq\label{dq}
\ga(y)= (y,q(y)) \qquad \text{for a homogeneous polynomial $q$ of degree $d_q \geq 1$.}
\eeq
We now denote $f$ as a function of $(x,t) \in \R^{n+1}$, so that along the hypersurface $\ga$, we study the function $f(x-y,t-Q(y))$.
We continue to assume homogeneous polynomials $p_1,\ldots, p_M$ have been fixed as in (\ref{P_span}), but do not yet need further assumptions about them.

\subsection{A flat analogue: the $J$-operator}
The square function compares the $I$-operator defined in (\ref{I_dfn}) to  a smoother ``flat'' analogue (with a compatible nonisotropic scaling), defined as follows. 
Fix a $C^1$ function $\zeta$ with $\|\zeta\|_{C^1} \leq 1$, and with $\zeta$ supported on $B_1( \R)$. 
 Then for any $P_\lam \in \Pcal=\Span\{p_1,\ldots, p_M\}$, any $C^1$ bump function $\eta$ supported in the unit ball, and any $a>0$, define the  operator, acting on a Schwartz function $f$, by
 \beq\label{J_dfn}
 ^{(\eta)}J_a^\lam f(x,t) = \int_{\R^{n+1}} f(x-y,t-s) e^{i P_\lam(y/a)} \frac{1}{a^n} \eta( \frac{y}{a}) \frac{1}{a^{2}} \zeta(\frac{s}{a^{2}}) dy ds.
 \eeq
The comparison to the operator $^{(\eta)}I_a^\lam$ is clear if we write  
 \[ ^{(\eta)}I_a^\lam f(x,t) = \int_{\R^{n+1}} f(x-y,t-s) e^{i P_\lam(y/a)} \frac{1}{a^n} \eta( \frac{y}{a}) \del_{s=Q(y)}  dy ds.\]
 From now on, we will study the $J$-operator defined with a  choice of bump function $\zeta$ such that
 \beq\label{zeta_integral}
 \int_\R \zeta(s) ds=1.
 \eeq
 This is useful, because if we temporarily let $I_a^\lam(y,s)$ and $J_a^\lam(y,s)$ represent the kernels of the operators, respectively, then for each fixed $\lam \in \R^M$, $a>0$ and $y \in \R^n$, we have
 \[ \int_{\R}(I_a^\lam(y,s) - J_a^\lam(y,s))ds=0.\]
(This will be used in (\ref{I_uses_zero}) below, to introduce a derivative to a Littlewood-Paley projection, enabling the application of Lemma \ref{lemma_uDel} (iii); this ultimately leads to (\ref{IJ_cases_IJ}).)

\subsection{Littlewood-Paley decomposition}
In order to introduce the square function built from differences of the $I$-operator (\ref{I_dfn}) and the $J$-operator (\ref{J_dfn}), we will employ a Littlewood-Paley decomposition constructed in \cite{PieYun19}; we summarize its properties, all of which are proved explicitly in \cite[\S 4]{PieYun19}.
There exists a family of Schwartz functions on $\R$, denoted $\Delta_j(t)$  for $j \in \Z$,   with the following properties: $(\Del_j)\hat{\;}(t)$ is supported in an annulus where $|t| \approx 2^{-2j}$, 
\beq\label{Delta_cancel}
\int_\R \Del_j(t)dt=0, \qquad \text{for all $j$,}
\eeq
and $\sum_j (\Del_j)\hat{\;}(t)=1$ for all $t \neq 0$.
There is a second family of Schwartz functions on $\R$, denoted  $\tilde{\Delta}_j$, with $(\tilde{\Del}_j)\hat{\;}(t)$ supported on a slightly wider annulus $|t| \approx 2^{-2j}$ and with $(\tilde{\Del}_j)\hat{\;}(t)\con 1$ on the support of $(\Del_j)\hat{\;}(t)$ for each $j$.  This family of Schwartz functions satisfies the analogue of (\ref{Delta_cancel}) and $\sum_j (\Del_j)\hat{\;}(t)=1$ for all $t \neq 0$.
Define the operator, also denoted $\Del_j$, acting on any Schwartz function $f$ on $\R^{n+1}$ by
\[ \Del_j f(x,t) = \int_{\R^n} f(x,t-s)\Del_j(s)ds,\]
and define the operator $\tilde{\Del}_j$ analogously.

For $P_j=\Del_j$  or $P_j = \Del_j \tilde{\Del}_j$, for every $f \in L^2$,
\[ f = \sum_{j=-\infty}^\infty P_j f, \]
with convergence of the partial sums holding in the $L^2$ sense, and 
\beq\label{Del_sum_f}
\| (\sum_{j=-\infty}^\infty |P_jf|^2)^{1/2}\|_{L^2} \ll \|f\|_{L^2}.
\eeq
Consequently, if we define the operator 
\[ L_Nf = \sum_{|j| \leq N} \Del_j \tilde{\Del}_j f,\]
acting on $f$ of Schwartz class, then $L_N f$ is a Schwartz function, and $L_N f$ converges to $f$ in $L^2$ norm as $N \maps \infty$. This concludes the summary of the properties cited from \cite[\S 4]{PieYun19}.

We now fix $N$, so that all sums over $j$ below are finite. Since the estimates below are independent of $N$, we can take $N \maps \infty$ at the end of the argument. In particular, 
\[\|\sup_{\bstack{r \leq \|\lam\|< 2r}{k \in \Z}} | ^{(\eta_k)}I_{2^k}^\lam f|\; \|_{L^2}  
\leq \|\sup_{\bstack{r \leq \|\lam\|< 2r}{k \in \Z}}| ^{(\eta_k)}I_{2^k}^\lam L_N f|\; \|_{L^2}   + \| \sup_{\bstack{r \leq \|\lam\|< 2r}{k \in \Z}}|^{(\eta_k)}I_{2^k}^\lam (f-L_Nf)|\;\|_{L^2}  .
\]
By applying the trivial $L^2$ bound of (\ref{I_trivial_p}) to the last term, and the fact that $L_N f$ converges to $f$ in $L^2$ norm, we see that if we can prove
\beq\label{IP_ineq_LN}
 \|\sup_{\bstack{r \leq \|\lam\|< 2r}{k \in \Z}}| ^{(\eta_k)}I_{2^k}^\lam L_N f|\; \|_{L^2} \ll r^{-\del}\|f\|_{L^2},
\eeq
uniformly in $N$, 
then the key hypothesis (\ref{IP_ineq}) of Theorem \ref{thm_I_to_R} immediately follows. Thus our main goal is to prove (\ref{IP_ineq_LN}), for all $r \geq 1$, uniformly in $N$.

\subsection{The square function}
Fix a family $\{\eta_k \}_{k \in \Z}$ of $C^1$ bump functions supported in the unit ball $B_1(\R^n)$ with $C^1$ norm uniformly bounded by $1$. The key to proving (\ref{IP_ineq_LN}) is a square function, defined  each $r \geq 1$, for $f$ of Schwartz class,  by
\[ S_r(f)  = ( \sum_{k \in \Z} ( \sup_{r \leq \|\lam\|< 2r} | ( ^{(\eta_k)}I_{2^k}^\lam - ^{(\eta_k)}J_{2^k}^\lam) L_N f| )^2 )^{1/2}.\]
We aim to prove that there exists $\del>0$ such that for any Schwartz function $f$ on $\R^{n+1}$, and any $r \geq 1$,
\beq\label{S_ineq_0}
\| S_r(f)\|_{L^2(\R^{n+1})} \ll r^{-\del} \|f\|_{L^2(\R^{n+1})},
\eeq
uniformly in $N$. Let us see why this implies (\ref{IP_ineq_LN}), and hence
  (\ref{IP_ineq}).
First note that for any $r \geq 1$, 
\beq\label{IJS}
\sup_{\bstack{r \leq \|\lam\|< 2r}{k \in \Z}}| ^{(\eta_k)}I_{2^k}^\lam L_N f| \leq \sup_{\bstack{r \leq \|\lam\|< 2r}{k \in \Z}}| ^{(\eta_k)}J_{2^k}^\lam L_N f| + S_r(f).
\eeq
Thus if we have proved (\ref{S_ineq_0}), to deduce (\ref{IP_ineq_LN}) it suffices to bound the first term on $L^2$ with decay in $r$. For this we cite the following,  an immediate consequence of the work of Stein and Wainger in \cite[Thm. 1]{SWCarl}:
 \begin{thm}\label{thm_main_J}
 Let $p_1(y),\ldots, p_M(y)$ be a fixed set of linearly independent, homogeneous polynomials in $\R[X_1,\ldots,X_n]$, each of degree at least 2.  
For any family $\{\eta_k \}_{k \in \Z}$ of $C^1$ bump functions supported in the unit ball $B_1(\R^n)$ with $C^1$ norm uniformly bounded by $1$, there exists a fixed $\del>0$ such that for any Schwartz function $f$ on $\R^{n+1}$ and any $r \geq 1$, 
\beq\label{JP_ineq}
\| \sup_{\bstack{\lam \in \R^M}{r \leq \| \lambda \| < 2r}} \sup_{k \in \Z} | ^{(\eta_k)}J_{2^k}^\lam f(x,t)| \|_{L^2(\R^{n+1})} \leq
	A r^{-\del} \|f\|_{L^2(\R^{n+1})},
\eeq
in which the norm $A$ may depend on $p, n,$ and the fixed polynomials $p_1,\ldots,p_M$.
 \end{thm}
 Applying this theorem to the first term in (\ref{IJS}), and (\ref{S_ineq_0}) to the second term, then implies (\ref{IP_ineq_LN}) and hence (\ref{IP_ineq}).

 Theorem \ref{thm_main_J} follows from Stein and Wainger's original work with only a few minimal observations. The case  we apply to prove Theorem \ref{thm_main_R} is explicitly proved in \cite[Thm. 6.1]{PieYun19}, so we do not repeat the proof. (The same proof works verbatim for any $d_q \geq 1$ in the notation of (\ref{dq}).) 
 We will, however, explicitly state   an intermediate step from Stein and Wainger's proof,   because we will apply it again in \S \ref{sec_TT_to_K} (see (\ref{avg_op_2})). We recall from \cite{SWCarl} that their
  main strategy to prove Theorem \ref{thm_main_J} is to linearize and consider an operator $Tf(x,t)= {^{(\eta_k)}J}_{2^k}^\lam f(x,t)$ for measurable stopping-time functions $k(x,t): \R^{n+1} \maps \Z$ and $\lam(x,t): \R^{n+1} \maps \R^M$. Then the key is to prove that  $TT^*$ is bounded on $L^2$ with a norm that decays in $r$, as long as $\lam(x,t)$ takes values with $\|\lam\| \approx r.$ To prove this, Stein and Wainger compute the kernel of $TT^*$, and prove that it is compactly supported and has an upper bound that exhibits decay in $r$ (in an appropriate sense). 
  
  More precisely, the kernel of $TT^*$ takes the following form. Fix homogeneous polynomials $p_1,\ldots, p_M$  in $\R[X_1,\ldots,X_n]$.
  Let $\Psi(u,z)$ be a $C^1$ bump function supported in $B_2 \times B_1 \subset \R^n \times \R^n$, and define
  \[ \| \Psi \|_{C^1} := \sup_{(u,z) \in B_2(\R^n) \times B_1(\R^n)} (| \Psi(u,z)| + |\nabla_z\Psi(u,z)|).\]
  Define for stopping-times $\nu,\mu$ taking values in $\R^M$ the function
  \[ K_\flat^{\nu,\mu} (u) = \int_{\R^n} e^{i P_\nu(u+z) - iP_\mu(z)}\Psi(u,z) dz,\]
  where $P_\nu(y) = \sum_{j=1}^M \nu_j p_j(y)$ and $P_\mu(y) = \sum_{j=1}^M \mu_j p_j(y)$. 
  Trivially, $|K_\flat^{\nu,\mu}(u)| \ll \chi_{B_2}(u)\|\Psi\|_{C^1}$.
The following nontrivial bound holds: 
 
 \begin{prop}\label{prop_K_flat}
  Let $p_1(y),\ldots, p_M(y)$ be a fixed set of linearly independent, homogeneous polynomials in $\R[X_1,\ldots,X_n]$, each of degree at least 2.  
  Let $\Psi(u,z)$ be a $C^1$ bump function supported in $B_2 \times B_1 \subset \R^n \times \R^n$ with 
  $\|\Psi\|_{C^1} \leq 1$.
There exists a fixed $\del>0$ such that for any $r \geq 1$,  if $\nu,\mu$ satisfy 
\[ r \leq \| \nu\|,\|\mu\| \leq 2r\]
then there exists a measurable set $G^\nu \subset B_2(\R^n)$ (which depends only on $\nu$, and not on $\mu$ or $\Psi$), such that $|G^\nu| \ll r^{-\del}$ and 
\[ |K_\flat^{\nu,\mu}(u)| \ll r^{-\del} \chi_{B_2}(u) + \chi_{G^\nu}(u).\]
All implied constants are dependent only on $n$ and the fixed polynomials $p_1,\ldots,p_M$, and are independent of $\nu,\mu, \Psi$.
 \end{prop}
 This is simply a reformulation of \cite[Lemma 4.1]{SWCarl} as  stated in \cite[Prop. 3.4]{PieYun19}. 
 It is crucial in Stein and Wainger's proof by the $TT^*$ method that none of the polynomials $p_j$ contains a linear term, and this restriction thus applies in Proposition \ref{prop_K_flat} as well as in our main theorem.

\subsection{Strategy to prove the square function estimate}
We have established that to prove (\ref{IP_ineq_LN}), the key is to prove (\ref{S_ineq_0}). 
As a final result of this section, we record the strategy to prove the square function estimate (\ref{S_ineq_0}):
\begin{thm}\label{thm_S_to_IJ}
Suppose that for any family $\{\eta_k \}_{k \in \Z}$ of $C^1$ bump functions supported in the unit ball $B_1(\R^n)$ with $\|\eta_k\|_{C^1} \leq 1$, there exists a fixed $\del_0>0$ and $\ep_0>0$ such that for any Schwartz function $F$ on $\R^{n+1}$ and any $r \geq 1$,
\beq\label{S_ineq_prequel}
\| \sup_{r \leq \|\lam\|< 2r} | ( ^{(\eta_k)}I_{2^k}^\lam - ^{(\eta_k)}J_{2^k}^\lam) \Del_j F| \; \|_{L^2} \ll r^{-\del_0} 2^{-\ep_0|j-k|}\|F\|_{L^2}, \qquad \text{for all $j,k \in \Z$}.
\eeq
Then for any such family $\{\eta_k\}$, there exists $\del>0$ such that for any Schwartz function $f$ on $\R^{n+1}$, and any $r \geq 1$,
\beq\label{S_ineq}
\| S_r(f)\|_{L^2(\R^{n+1})} \ll r^{-\del} \|f\|_{L^2(\R^{n+1})},
\eeq
uniformly in $N$, 
and hence (\ref{IP_ineq}) holds.
\end{thm}
 Indeed, we recall that 
\[ |S_r(f)|^2 \leq \sum_{k \in \Z} ( \sum_{|j| \leq N} \sup_{\|\lam\| \approx r} |( ^{(\eta_k)}I_{2^k}^\lam - ^{(\eta_k)}J_{2^k}^\lam) \Del_j \tilde{\Del}_j f|)^2.
\]
By inserting $2^{-\ep_1|j-k|}2^{+\ep_1|j-k|}$ for some $\ep_1<\ep_0$, and applying Cauchy-Schwarz to the sum over $j$, we see that 
\[ |S_r(f)|^2 \leq \sum_{k \in \Z} ( \sum_{|j| \leq N}2^{-2\ep_1|j-k|})( \sum_{|j| \leq N}2^{2\ep_1|j-k|}\sup_{\|\lam\| \approx r} |( ^{(\eta_k)}I_{2^k}^\lam - ^{(\eta_k)}J_{2^k}^\lam) \Del_j \tilde{\Del}_j f|^2).
\]
The first sum over $j$ is $\ll 1$ uniformly in $k$. Taking $L^2$ norms,  we apply the hypothesis (\ref{S_ineq_prequel}) to each $F=\tilde{\Del}_j f$ to see that 
\[ \| S_r(f)\|^2_{L^2} \ll r^{-2\del_0}\sum_{k \in \Z} \sum_{|j| \leq N} 2^{2\ep_1|j-k|}  2^{-2\ep_0|j-k|} \|\tilde{\Del}_jf\|^2_{L^2}.\]
We use the fact that $\ep_1<\ep_0$ to sum over $k$, and apply the property (\ref{Del_sum_f}) of the projections $\tilde{\Del}_j$ to conclude that $\|S_r(f)\|_{L^2}^2 \ll r^{-2\del_0} \|f\|_{L^2}^2$, and the theorem is proved. 

The rest of the paper will focus on proving the hypothesis (\ref{S_ineq_prequel}).
Note that we now need to prove (\ref{S_ineq_prequel}) for each fixed $k$ rather than for an operator involving a supremum over $k$; fixing $k$ is a key advantage we have gained by working with the square function. (The later utility of having $k$ be fixed is visible in (\ref{avg_op}), and subsequent similar places.)

\section{Step 3: Change of variables to isolate the kernel $K_{\sharp,l}^{\nu,\mu}$}\label{sec_change_var}

Our strategy to prove the main hypothesis (\ref{S_ineq_prequel}) in Theorem \ref{thm_S_to_IJ}
is to apply the $TT^*$ method to an operator $T$ that is (roughly) of the form 
\[f\mapsto \sup_{r \leq \|\lam\|< 2r} | ( ^{(\eta_k)}I_{2^k}^\lam - ^{(\eta_k)}J_{2^k}^\lam) \Del_j f|.\]
To do so precisely, we will linearize this operator by replacing the supremum over $\lam$ by a stopping-time function $\lam(x,t)$ taking values with $\|\lam(x,t)\| \in [r,2r)$, and then we will compute $TT^* f(x,t)$ for a Schwartz function $f$. Our goal is to write $TT^*f(x,t)$ as a convolution of $f$ with a kernel, say $K$, that is itself defined by an oscillatory integral.

\begin{remark}\label{remark_ngeq2}
Since $TT^*f(x,t)$ is an integral over $\R^{2n}$, and a convolution of $f(x,t)$ with a kernel is an integral over $n+1$ variables, at least formally  we would expect $K$ to be an integral over $n-1$ variables. This is ultimately why the present paper is limited to the consideration of $n \geq 2$.
\end{remark}

  From now on, the fact that we work with $\ga(y) = (y,Q(y))$ with $Q(y)$ quadratic becomes crucial to our method, and in particular we assume from now on that 
\beq\label{Q_dfn}
Q(y) = \sum_{i=1}^n \theta_i y_i^2, \qquad \theta_i \in \{\pm 1\}.\eeq
Recall from \S \ref{sec_prelim} that this suffices to prove Theorem \ref{thm_main_R}.

 At least intuitively, if $Tf$ integrates $f(x-y,t-Q(y))$ over $y \in \R^n$, then we expect $TT^*f(x,t)$ to integrate $f(x-(y-z),t-(Q(y) - Q(z)))$ over $y \in \R^n, z \in \R^n$. It is then natural to set $u=y-z$, so we face an integral of $f(x-u,t-(Q(u+z)-Q(z)))$ over $u \in \R^n, z \in \R^n$. To isolate the kernel $K$, we must ``free'' $n-1$ variables  so that only $n+1$ coordinates appear in the argument of $f$, that is, so that $Q(u+z)-Q(z)$ only involves $n+1$ coordinates. This motivates the change of variables we now describe, which crucially uses the quadratic behavior of $\ga(y)$. It is an interesting open problem to develop an argument for more general hypersurfaces.
 
For $Q$ as in (\ref{Q_dfn}), 
define for any $u\in\R^n$, that
\[ \tilde{u}=(\theta_1u_1,\ldots,\theta_n u_n),\]
where the signs $\theta_i$ are determined once and for all by (\ref{Q_dfn}).
This notation will be used extensively for the remainder of the paper.
Then define
\[\langle u,z\rangle=\sum_{i=1}^n\theta_iu_iz_i = \tilde{u}\cdot z.\]
Consequently,
\[ Q(u+z)-Q(z)=Q(u)+2|u|\frac{\langle u,z\rangle}{|u|}.\]
This motivates us to replace $z$ by defining new coordinates $\tau, \sig$, where   $\tau \in \R$, with $\tau =  \langle u,z\rangle/|u|$, and $\sig \in \R^{n-1}$. To define $\sig$ appropriately, we require a notational convention: for $z\in\R^n$ and $l\in\{1,\ldots,n\}$, define  
\[ z^{(l)} = (z_1,\ldots,\hat{z_l},\ldots,z_n)\in\R^{n-1}, \quad \text{that is, omitting the $l$-th coordinate.}\]

Now for a fixed $l\in \{1,\ldots,n\}$ we can define the change of variables $z\mapsto(\tau,\sigma) \in \R \times \R^{n-1}$ by
\begin{equation}\label{subdef} \tau=\frac{\langle u,z\rangle}{|u|} \quad\text{and}\quad\sigma=\frac{\tau\tilde{u}^{(l)}-|u|z^{(l)}}{\theta_l u_l}.\end{equation}
The formula for $z$ in terms of $(\tau,\sigma)$ is
\begin{equation}\label{zformula} z^{(l)}=\frac{\tau\tilde{u}^{(l)}-\theta_lu_l\sigma}{|u|}\quad\text{and}\quad z_l=\frac{\theta_lu_l\tau+\tilde{u}^{(l)}\cdot\sigma}{|u|}, \end{equation}
or equivalently,
\[ \begin{bmatrix}
z^{(l)}\\ z_l
\end{bmatrix}=\frac{1}{|u|}\begin{bmatrix}
-\theta_lu_l\text{Id}_{n-1}&\tilde{u}^{(l)}\\
(\tilde{u}^{(l)})^T&\theta_l u_l
\end{bmatrix}\begin{bmatrix}
\sigma\\ \tau
\end{bmatrix} .\]

It will be important in what follows that if $(u,z)$ lies in $B_2 \times B_1 \subset \R^n \times \R^n$, then $(\tau,\sig)  \in \R^n$ also lies in a compact region. To achieve this, we will use the fact that for any $u \in \R^n$, we can choose for which $l$ we perform the above change of variables. That is, we fix a small $c_0>0$ once and for all, and choose a partition of unity 
\beq\label{W_partition}
1 = \sum_{1 \leq l \leq n} W_l (s), \qquad s \in S^{n-1},
\eeq
where each $W_l \in C_c^\infty(S^{n-1})$ is supported where $|s_l|/|s| \geq c_0$. For $u$ such that $u/|u|$ lies in the support of $W_l$, it follows that $|u_l|/|u| \geq c_0$, and we can apply the change of variables (\ref{subdef}) for this choice $l$.

\subsection{Definition of $K_{\sharp,l}^{\nu,\mu}$ kernel}
Now, with this key change of variables in mind, we can define the oscillatory integral that is the heart of the matter. Let $\Psi(u,z)$ be a $C^1$ bump function supported in $B_2 \times B_1 \subset \R^n \times \R^n$.
Define,  for $u \in \R^n$ and for $1 \leq l \leq n$ such that $u/|u|$ lies in the support of $W_l$, the oscillatory integral
\beq\label{K_sharp_dfn_Psi} K_{\sharp,l}^{\nu,\mu}(u,\tau)=\int_{\R^{n-1}}e^{iP_\nu(u+z)-iP_\mu(z)}\Psi(u,z)d\sigma,
\eeq
where $z$ is defined implicitly in terms of $u, \tau, \sig$ as in (\ref{subdef}). (For $u \in B_2$ such that $u/|u|$ lies outside the support of $W_l$, by convention we set $K_{\sharp,l}^{\nu,\mu}(u,\tau) \con 0$ for $\tau \in B_1$, so the upper bounds we prove below for $K_{\sharp,l}^{\nu,\mu}$ will hold for all $(u,\tau) \in B_2 \times B_1$.) Here $\nu,\mu$ are stopping-times taking values in $\R^M$, which will be described momentarily.
Given a function $\Psi(u,z)$ as above, we   use the notation
\beq\label{Psi_norm}
\| \Psi\|_{C^1(\sig)} := \sup_{(u,z) \in B_2(\R^n) \times B_1(\R^n)} ( |\Psi(u,z)| + |\nabla_\sig \Psi(u,z)|),
\eeq
which is well-defined, for $z$ defined in terms of $u,\tau,\sig$ by (\ref{subdef}).

We now check that if $u, z$ are each compactly supported in $B_2 \times B_1 \subset \R^n \times \R^n$ and $u/|u|$ lies in the support of $W_l$, then  $K_{\sharp,l}^{\nu,\mu}(u,\tau)$ is compactly supported with respect to $(u,\tau) \in \R^n \times \R$, and the region of integration over $\sig \in \R^{n-1}$ in (\ref{K_sharp_dfn_Psi}) is compactly supported.
To see this, note that
\begin{align*}
    |z|^2 &= |z^{(l)}|^2 +z_l^2 
     = |u|^{-2}[|\tau\tilde{u}^{(l)}-\theta_lu_l\sigma|^2 +|\theta_lu_l\tau+\tilde{u}^{(l)}\cdot\sigma|^2]\\
   & = |u|^{-2} [\sum_{i \neq l} (\tau \theta_i u_i - \theta_l u_l \sig_i)^2 + (\tau \theta_l u_l + \tilde{u}^{(l)}\cdot\sigma)^2]\\
   & = |u|^{-2} [\sum_{i \neq l} (\tau^2  u_i^2 - 2\tau \theta_i u_i \theta_l u_l \sig_i+  u_l^2 \sig_i^2) + \tau^2 u_l^2 + 2 \tau \theta_l u_l (\tilde{u}^{(l)}\cdot\sigma) + (\tilde{u}^{(l)}\cdot\sigma)^2].
\end{align*}
We conclude that
\[ |z|^2 = \tau^2 + (u_l/|u|)^2|\sig|^2 + |u|^{-2}(\tilde{u}^{(l)}\cdot\sigma)^2.\]
Since all terms in this relation are non-negative, it follows that $\tau^2 \leq |z|^2 \leq 1$. We also see that 
\beq\label{sig_bound}
|\sig|^2 \leq (u_l/|u|)^{-2} \leq c_0^{-2},
\eeq
since $u/|u|$ is in the support of $W_l$. 
In particular, since $\Psi(u,z)$ has support in $B_2 \times B_1 \subset \R^n \times \R^n$, this proves the trivial upper bound 
\beq\label{K_sharp_trivial}
|K_{\sharp,l}^{\nu,\mu}(u,\tau)| \ll_{c_0} \chi_{B_2}(u)\chi_{B_1}(\tau).
\eeq

Once we prove a far stronger bound  for  $K_{\sharp,l}^{\nu,\mu}$, one that exhibits decay  in $r$ when $\|\mu\|,\|\nu\| \approx r$, then we can complete the proof of the key hypothesis of Theorem \ref{thm_S_to_IJ}, and hence complete the proof of the main theorem of this paper. This deduction is the content of the next section.

\section{Step 4: Reduction of the square function estimate to bounding $K_{\sharp,l}^{\nu,\mu}$} \label{sec_TT_to_K}
 
 The main result of this section is the following theorem:
\begin{thm}\label{thm_K_to_S}
Let $\Psi(u,z)$ be a $C^1$ function supported on $B_2 \times B_1 \subset \R^n \times \R^n$, with 
$\|\Psi\|_{C^1(\sig)} \leq 1$ as in (\ref{Psi_norm}).
For each $1 \leq l \leq n$, for $u\in \R^n$ such that $u/|u|$ lies in the support of $W_l$, consider the kernel $K_{\sharp,l}^{\nu,\mu}$ as defined in (\ref{K_sharp_dfn_Psi}).
Suppose that there exists a small $\del>0$ such that for all $r \geq 1$, if the stopping times $\nu,\mu$ satisfy 
\[ r \leq \|\nu\|,\|\mu\| \leq 2r,\]
then there exists a measurable set $G^\nu \subset B_2(\R^n)$ (depending on $\nu$ but  independent of $\mu,r,\Psi$), and for each $u \in B_2(\R^n)$ there exists a measurable set $F_u^\nu \subset B_1(\R)$ (depending on $u$ and $\nu$ but independent of $\mu, r, \Psi$), such that 
\beq\label{K_small_sets_prequel}
|G^\nu| \ll r^{-\del}, \qquad |F_u^\nu| \ll r^{-\del},
\eeq
and 
\beq\label{K_sharp_theorem_ineq_prequel}
|K_{\sharp,l}^{\nu,\mu}(u,\tau)| \ll r^{-\del} \chi_{B_2}(u) \chi_{B_1}(\tau) + \chi_{G^\nu}(u) \chi_{B_1}(\tau) + \chi_{B_2}(u)\chi_{F_u^\nu}(\tau),
\eeq
in which all implicit constants depend only on $n$ and the fixed polynomials $p_1,\ldots, p_M$. 
Then the hypothesis (\ref{S_ineq_prequel}) of Theorem \ref{thm_S_to_IJ} holds.
\end{thm}
\begin{remark}\label{remark_mu}
In the proof of this theorem, it is crucial that in the hypothesis, the sets in (\ref{K_small_sets_prequel}) and  (\ref{K_sharp_theorem_ineq_prequel}) are completely independent of $\mu$, although they are allowed to depend on $\nu$. This is because when we apply these bounds within the proof of Theorem \ref{thm_K_to_S} in the steps (\ref{Ksharp_app1}) and (\ref{Ksharp_app2}) below, $\nu=\lam(x,t)$ does not depend on variables of integration, while $\mu = \lam(x-u,t-\theta)$ does. Consequently, a dependence on $\nu$ can be controlled by an averaging operator (Lemma \ref{lemma_average}), while a dependence on $\mu$ cannot. Proving that such sets $G^\nu$ and $F_u^\nu$ do exist, independent of $\mu$, was a key difficulty in the previous work \cite{PieYun19} and is also the main difficulty of this paper; it is handled when we prove Theorem \ref{thm_main_K}.
\end{remark}

 After we prove   Theorem \ref{thm_K_to_S} in this section, all that remains to complete the proof of our main theorem is to show that the estimate (\ref{K_sharp_theorem_ineq_prequel}) for $K_{\sharp,l}^{\nu,\mu}$ holds.
 
 To prove that (\ref{S_ineq_prequel}) holds, as claimed in Theorem \ref{thm_K_to_S}, it suffices to prove three bounds, under the hypotheses of the theorem:
\begin{align}
 \|\sup_{\|\lam\|\approx r} |(^{(\eta_k)}I_{2^k}^\lam - ^{(\eta_k)}J_{2^k}^\lam) \Del_j F|\,\|_{L^2} & \ll r^{-\del_0}2^{-\ep_0(j-k)}\|F\|_{L^2} \qquad \text{for $j \geq k$} \label{IJ_cases_IJ}\\
 \|\sup_{\|\lam\|\approx r} |^{(\eta_k)}I_{2^k}^\lam   \Del_j F|\,\|_{L^2} & \ll r^{-\del_0}2^{\ep_0(j-k)}\|F\|_{L^2}  \qquad \text{for $j < k$} \label{IJ_cases_I}\\
 \|\sup_{\|\lam\|\approx r} |^{(\eta_k)}J_{2^k}^\lam   \Del_j F|\,\|_{L^2} & \ll r^{-\del_0}2^{\ep_0(j-k)}\|F\|_{L^2}  \qquad \text{for $j <k$.} \label{IJ_cases_J}
\end{align}
The bound (\ref{IJ_cases_J}) for the smoother ``flat'' operator $J_{2^k}^\lam$ has already been proved in full by \cite[Prop. 7.2]{PieYun19}.
 We will prove (\ref{IJ_cases_IJ}) and (\ref{IJ_cases_I}) by the method of $TT^*$.
 While we give a complete proof, we do not elaborate on the motivation for each step, for which we refer the reader to the more extensive exposition in \cite{PieYun19}.

  \subsection{Computing the kernel of $TT^*$ to prove (\ref{IJ_cases_IJ})}
  Fix $r \geq 1$. Fix $k \in \Z$ and set $a=2^k$. We also fix a bump function $\eta_k$, and for simplicity suppress it in the notation, as follows. Let $\lam(x,t)$ be a measurable  stopping-time function that takes values in 
$r \leq \lam(x,t) \leq 2r$. 
For each $j\in \Z$, define the operator 
\[T=  (^{(\eta_k)}I_{2^k}^\lam - ^{(\eta_k)}J_{2^k}^\lam) \Del_j = (I_{2^k}^\lam -  J_{2^k}^\lam) \Del_j,\]
so that
\[ Tf(x,t)=\int_{\R^{n+2}}f(x-y,t-\theta)e^{iP_{\lambda(x,t)}(\frac{y}{a})}\frac{1}{a^n}\eta(\frac{y}{a})\big[\delta_{s=Q(y)}-\frac{1}{a^2}\zeta(\frac{s}{a^2})\big]\Delta_j(\theta-s)dydsd\theta .\]
The dual $T^*$ is 
\[ T^*g(x,t)=\int_{\R^{n+2}}g(x+z,t+\w)e^{-iP_{\lambda(x+z,t+\w)}(\frac{z}{a})}\frac{1}{a^n}\eta(\frac{z}{a})\big[\delta_{\xi=Q(z)}-\frac{1}{a^2}\z(\frac{\xi}{a})\big]\Delta_j(\w-\xi)dzd\xi d\w. \]
It then follows that
\begin{align*} 
TT^*f(x,t)&= \int_{\R^{2n+4}}f(x-y+z,t-\theta+\w)e^{i[P_{\lambda(x,t)}(\frac{y}{a})-P_{\lambda(x+z-y,t-\theta+\w)}(\frac{z}{a})]}\frac{1}{a^{2n}}\eta(\frac{y}{a})\eta(\frac{z}{a})\\
&\qquad\qquad \cdot\big[\delta_{s=Q(y)}-\frac{1}{a^2}\z(\frac{s}{a^2})\big]\big[\delta_{\xi=Q(z)}-\frac{1}{a^2}\z(\frac{\xi}{a^2})\big]\Delta_j(\theta-s)\Delta_j(\w-\xi)dzd\xi d\w dydsd\theta.
\end{align*} 
As a first step, to verify (\ref{IJ_cases_IJ}) under the hypotheses of Theorem \ref{thm_K_to_S}, we need to show that 
\beq\label{TTstar_goal}
\| TT^*f\|_{L^2} \ll r^{-2\del_0} 2^{-2\ep_0|j-k|}\|f\|_{L^2}.
\eeq

To study $TT^*$, we require further   Littlewood-Paley projections closely related to $\Del_j$ by taking convolutions, antiderivatives, or derivatives. We quote the key properties from \cite[Lemmas 4.3, 4.4, 4.5]{PieYun19}, which we summarize as a lemma:

\newcommand{\uDel}{\underline{\Del}}
\begin{lemma}\label{lemma_uDel}
(i) Convolutions: Define \[ \underline{\Del}_j(t) = \int_\R \Del_j(w+t)\Del_j(w)dw.\]
For each $j \in \Z$,   
\beq\label{uDel_cancel}
\int_{\R} \uDel_j(t)dt=0.
\eeq
For each $a=2^k$,
$\uDel_j(t) = \frac{1}{a^2} \uDel_{j-k}(\frac{t}{a^2})$. 
Uniformly in $j$, $\|\uDel_j\|_{L^1} \ll 1$.

(ii) Antiderivatives: There exist Schwartz functions $\tilde{\Delta}$ and $\tilde{\uDel}$ on $\R$ such that upon defining $\tilde{\Del}_j(t) = 2^{-2j} \tilde{\Delta}(2^{-2j}t)$ and $\tilde{\uDel}_j(t) = 2^{-2j} \tilde{\uDel}(2^{-2j}t)$, then  
\beq\label{antiderivs}
\Del_j(t) = 2^{2j}(\frac{d}{dt} \tilde{\Del}_j)(t), \qquad 
\uDel_j(t) = 2^{2j}(\frac{d}{dt} \tilde{\uDel}_j)(t), \qquad j \in \Z.
\eeq
Finally, $\|\tilde{\Del}_j\|_{L^1} \ll 1$ and $\|\tilde{\uDel}_j\|_{L^1}\ll 1$ uniformly in $j$. 

(iii) Mean value: Define $\psi(t) = (1+t^2)^{-1}$. For each $j \in \Z$ define   $\psi_j(t) = 2^{-2j} \psi(2^{-2j}t)$, so that $\psi_j$ is a non-negative function on $\R$ with $\| \psi_j \|_{L^1} \ll 1, $ uniformly in $j$.   Then $|\uDel'(t)| \ll 2^{-2j} \psi_j(t)$, and for any $|\xi| \leq 2$, 
$|\uDel_j(t + \xi) - \uDel_j(t)| \ll 2^{-2j}|\xi| \psi_j(t)$.
\end{lemma}
The dilations employed in the lemma are intended to be compatible with parabolic rescaling. We will apply (iii) in \S \ref{sec_I_term} and \S \ref{sec_II_term} to prove (\ref{IJ_cases_IJ}), and (i) and (ii) in \S \ref{sec_I_I_only} to prove (\ref{IJ_cases_I}).

A short computation using the definition of $\uDel_{j-k}$ (entirely analogous to \cite[(7.10)]{PieYun19}), shows that the operator $TT^*$ given above can be written as
\beq\label{TTK}
TT^*f(x,t)=\int_{\R^{n+1}}f(x-u,t-\theta)\frac{1}{a^{n+2}}K^{\lambda(x,t),\lambda(x-u,t-\theta)}(\frac{u}{a},\frac{\theta}{a^2})dud\theta
\eeq
where for each $\nu,\mu\in\R^{d-1}$, we define the kernel 
\begin{align*} K^{\nu,\mu}(u,\theta)&=\int_{\R^{n+2}}e^{iP_\nu(u+z)-iP_\mu(z)}\eta(u+z)\eta(z)\\
&\qquad\cdot\big[\delta_{s=Q(u+z)}-\zeta(s)\big]\big[\delta_{\xi=Q(z)}-\z(\xi)\big]\underline{\Delta}_{j-k}(\theta-s+\xi)dzd\xi ds\\
&=\int_{\R^{n+1}}e^{iP_\nu(u+z)-iP_\mu(z)}\eta(u+z)\eta(z)\big[\delta_{\xi=Q(z)}-\z(\xi)\big]\underline{\Delta}_{j-k}(\theta-Q(u+z)+\xi)dzd\xi \\
&\qquad +\int_{\R^{n+2}}e^{iP_\nu(u+z)-iP_\mu(z)}\eta(u+z)\eta(z)\zeta(s)\big[\delta_{\xi=Q(z)}-\z(\xi)\big]\underline{\Delta}_{j-k}(\theta-s+\xi)dzd\xi ds\\
&=: {\bf{I}}+{\bf{II}}. 
\end{align*} 
To verify (\ref{TTstar_goal}), it suffices to prove the bound separately, for the contribution of the term $\mathbf{I}$ to the kernel, and the contribution of $\mathbf{II}$.

\subsection{The contribution of the term $\mathbf{I}$}\label{sec_I_term}
To analyze the contribution of  {\bf{I}}, we   write
\[ {\bf{I}}=\sum_{l=1}^n{\bf{I}}_l\]
where for each $1\le l\le n$,
\begin{multline*}
    {\bf{I}}_l= W_l\Bigl(\frac{u}{|u|}\Bigr)\int_{\R^{n}}e^{iP_\nu(u+z)-iP_\mu(z)}\eta(u+z)\eta(z)\underline{\Delta}_{j-k}(\theta-Q(u+z)+Q(z))dz  \\
   -W_l\Bigl(\frac{u}{|u|}\Bigr)\int_{\R^{n+1}}e^{iP_\nu(u+z)-iP_\mu(z)}\eta(u+z)\eta(z)\z(\xi)\underline{\Delta}_{j-k}(\theta-Q(u+z)+\xi)dzd\xi .
\end{multline*}
Using the fact that $\int\z(\xi+Q(z))d\xi=1$, this can be re-written as
\begin{multline}\label{I_uses_zero}
 \mathbf{I}_l= W_l\Bigl(\frac{u}{|u|}\Bigr)\int_{\R^{n+1}}e^{iP_\nu(u+z)-iP_\mu(z)}\eta(u+z)\eta(z)\z(\xi+Q(z))\\ \cdot\big(\underline{\Delta}_{j-k}(\theta-Q(u+z)+Q(z))-\underline{\Delta}_{j-k}(\theta-Q(u+z)+Q(z)+\xi)\big)dzd\xi
. 
\end{multline} 
For each fixed $1 \leq l \leq n$, we make the $l$-th  change of variables (\ref{subdef}), so that
\begin{multline*}
    \mathbf{I}_l= \Bigl(\frac{|u_l|}{|u|}\Bigr)^{n-2} W_l\Bigl(\frac{u}{|u|}\Bigr)\int_{\R^{2}}K_{\sharp,l}^{\nu,\mu}(u,\tau;\xi) \\
    \cdot \big(\underline{\Delta}_{j-k}(\theta-Q(u)-2|u|\tau)-\underline{\Delta}_{j-k}(\theta-Q(u)-2|u|\tau+\xi)\big) d\tau d\xi,\end{multline*}
in which 
\beq\label{Ksharp_app1} 
K_{\sharp,l}^{\nu,\mu}(u,\tau;\xi)
= 
\int_{\R^{n-1}}e^{iP_\nu(u+z)-iP_\mu(z)}\eta(u+z) \eta(z) \zeta(\xi + Q(z))d\sigma,
\eeq
where $z$ is defined implicitly in terms of $u, \tau, \sig$, and we recall from (\ref{TTK}) that
\beq\label{stopping_time_dep}
\nu =\lam(x,t), \qquad \mu = \lam(x-u,t-\theta)
\eeq
are stopping-times with values such that $r \leq \| \nu\|,\|\mu\| < 2r$.
 Note that the $\xi$-support is in $|\xi| \leq 2$.
By the mean-value property of Lemma \ref{lemma_uDel} (iii),
\[
    |{\bf{I}}_l|\ll 2^{-2(j-k)}\chi_{B_2}(u)\int_{\R^2} |K_{\sharp,l}^{\nu,\mu}(u,\tau;\xi)|\chi_{B_1}(\tau)\chi_{B_2}(\xi)\psi_{j-k}(\theta-Q(u)-2|u|\tau)d\tau d\xi .
\]
For each fixed $|\xi| \leq 2$, $K_{\sharp,l}^{\nu,\mu}(u,\tau;\xi)$ is a kernel of the form (\ref{K_sharp_dfn_Psi}) for 
\[\Psi(u,z) = \Psi_{\xi}(u,z) = \eta(u+z)\eta(z)\zeta(\xi+Q(z)).\]
We now apply the hypothesis of Theorem \ref{thm_K_to_S} to bound $K_{\sharp,l}^{\nu,\mu}(u,\tau;\xi)$ as in (\ref{K_sharp_theorem_ineq_prequel}). Note that the resulting bound is uniform in $\xi$, and so we can then integrate trivially over $\xi\in B_2(\R)$. We conclude that 
the contribution of $\mathbf{I}_l$ to $TT^*f$ in (\ref{TTK})   is  bounded by
\beq\label{avg_op}
\ll 2^{-2(j-k)} \iint_{\R^{n+1}} |f(x-u,t-\theta)|
  \chi( (x,t),(u,\theta))  du d\theta
  \eeq
     in which we define the function $\chi( (x,t),(u,\theta))$ to be
     \begin{multline*}
     \int_{\R} \frac{1}{a^n} \left\{r^{-\del} \chi_{B_2}(\frac{u}{a}) \chi_{B_1}(\tau)+  \chi_{G^{\lam(x,t)}}(\frac{u}{a}) \chi_{B_1(\tau)} +  \chi_{B_2}(\frac{u}{a})\chi_{F_{(\frac{u}{a})}^{\lam(x,t)}}(\tau)\right\} \cdot
     \\ \frac{1}{a^2} \psi_{j-k} \left( \frac{\theta - Q(u) - 2a|u| \tau}{a^2}\right) d\tau.
     \end{multline*}
We recall that $a=2^k$ is fixed, so that (\ref{avg_op}) is an averaging operator, not a maximal operator.  Next we will apply \cite[Lemma 7.5]{PieYun19}, which we quote:
\begin{lemma}\label{lemma_average}
Let $\chi(w,y)$ be an integrable function on $\R^m \times \R^m$ with 
\beq\label{chi_hyp}
\| \sup_{w \in \R^m}|\chi(w,y)|\,\|_{L^1(\R^m(dy))} \leq C, \qquad \quad \sup_{w \in \R^m}\| \chi(w,y)\|_{L^1(\R^m(dy))} \leq \lam.
\eeq
Then 
\[ \| \int_{\R^m} f(w-y) \chi(w,y) dy \|_{L^p(\R^m(dw))} \ll_{C,p} \lam^{1-1/p} \|f\|_{L^p(\R^m)}, \qquad 1 \leq p \leq \infty.\]
\end{lemma}
To verify the first condition in (\ref{chi_hyp}) for our choice of $\chi((x,t),(u,\theta))$ above, we simply note that by the hypothesis (\ref{K_small_sets_prequel}), the small sets satisfy $G^{\lam(x,t)}\subset B_2(\R^{n})$ and $F_{(\frac{u}{a})}^{\lam(x,t)} \subset B_1(\R)$ for all $(x,t)$, so that
 \beq\label{chi_L1_check}
 \sup_{(x,t)}|\chi((x,t),(u,\theta))|
    \ll  \int_{\R} \frac{1}{a^n}
\chi_{B_2}(\frac{u}{a}) \chi_{B_1}(\tau)
\frac{1}{a^2} \psi_{j-k} \left( \frac{\theta - Q(u) - 2a|u| \tau}{a^2}\right) d\tau.
\eeq
We compute the $L^1(dud\theta)$ norm of the above function of $(u,\theta)$ by integrating first in $\theta$, and we see that the first condition in (\ref{chi_hyp}) holds, uniformly in $j,k$. For the second condition in (\ref{chi_hyp}), we will use the small measures of the sets $G^{\lam(x,t)}$ and $F_{(\frac{u}{a})}^{\lam(x,t)}$. For each fixed $(x,t)\in \R^{n+1}$, we compute the $L^1(dud\theta)$ norm of $\chi((x,t),(u,\theta))$ by integrating in $\theta, \tau, u$ (in that order), and the second condition holds with $\lam = r^{-\del}$, by the hypothesis (\ref{K_small_sets_prequel}). Hence applying Lemma \ref{lemma_average} with $\lam = r^{-\del}$ and $p=2$, shows that the $L^2(dxdt)$ norm of the operator (\ref{avg_op}) is $\ll r^{-\del/2} 2^{-2(j-k)}$. This proves the conclusion of (\ref{IJ_cases_IJ}) for the contribution of $\mathbf{I}_l$, with $\del_0 = \del/4$ and $\ep_0=1$. Adding the contributions for $1 \leq l \leq n$ thus proves the conclusion of (\ref{IJ_cases_IJ}) for the contribution of $\mathbf{I}$.

\subsection{The contribution of the term $\mathbf{II}$}\label{sec_II_term}
To analyze the contribution of $\mathbf{II}$ to the kernel $K^{\nu,\mu}(u,\theta)$ of $TT^*$ in (\ref{TTK}), the argument is similar, but simpler, since we do not require the change of variables (\ref{subdef}) to study this term. A brief argument using $\int_{\R} \zeta(\xi + Q(z)) d\xi=1$ and linear changes of variables (entirely analogous to \cite[\S 7.3.2]{PieYun19}) shows that we may write the contribution of $\mathbf{II}$ to the kernel $K^{\nu,\mu}(u,\theta)$ as
\beq\label{II_int} \mathbf{II}= \iint_{\R^{n+2}} K_\flat^{\lam(x,t),\lam(x-u,t-\theta)}(u;\xi,s) (\uDel_{j-k}(\theta-s)  - \uDel_{j-k}(\theta-s+\xi)) ds d\xi,
\eeq
where we have defined for any $\nu,\mu \in \R^{d-1}$ the function
\[ K_\flat^{\nu,\mu}(u;\xi,s)
     = \int_{\R^{n}} e^{i P_\nu(u+z) - iP_\mu(z)}\eta(u+z)\eta(z)\zeta(\xi + Q(z)) \zeta (s+Q(z))dz.\]
Since $\eta$ and $\zeta$ are supported in $B_1(\R^n)$, it follows that the integral is over $z \in B_1(\R^n)$, and $K_\flat^{\nu,\mu}(u;\xi,s)$ is supported where $u \in B_2(\R^n)$, $\xi \in B_2(\R)$, $s \in B_2(\R)$. 
We apply the mean value of Lemma \ref{lemma_uDel} (iii) to bound the difference of $\uDel_{j-k}$ terms in (\ref{II_int}) by $\ll 2^{-2(j-k)} |\xi| \psi_{j-k}(\theta-s).$  We then apply  Proposition \ref{prop_K_flat} to bound $K_\flat^{\nu,\mu}(u;\xi,s)$, uniformly in $\xi,s \in B_2(\R)$, and integrate trivially over $\xi \in B_2(\R)$. Consequently, we see that the contribution of $\mathbf{II}$ to $TT^*f(x,t)$  in (\ref{TTK})  is bounded by 
\beq\label{avg_op_2}
\ll 2^{-2(j-k)} \iint_{\R^{n+1}} |f(x-u,t-\theta)|
  \chi( (x,t),(u,\theta))  du d\theta
\eeq
     in which we define the function $\chi( (x,t),(u,\theta))$ to be
\[
\chi( (x,t),(u,\theta)) = \frac{1}{a^n} (r^{-\del} \chi_{B_2}(\frac{u}{a}) + \chi_{G^{\lam(x,t)}}(\frac{u}{a})) \frac{1}{a^2} \int_{\R} \psi_{j-k}(\frac{\theta}{a^2}-s) \chi_{B_2}(s)ds. 
\]
We will again apply Lemma \ref{lemma_average} to bound the $L^2(\R^{n+1})$ norm of the averaging operator (\ref{avg_op_2}). The  first condition in (\ref{chi_hyp}) holds because 
\[ \sup_{(x,t)} |\chi( (x,t),(u,\theta)) |
    \leq \frac{1}{a^n}   \chi_{B_2}(\frac{u}{a})   \frac{1}{a^2} \int_{\R} \psi_{j-k}(\frac{\theta}{a^2}-s) \chi_{B_2}(s)ds.
\]
This has  bounded $L^1(dud\theta)$ norm, uniformly in $j,k,$ by integrating first in $\theta$ (using Lemma \ref{lemma_uDel} (iii)), then in $s$ and $u$.
The second condition in (\ref{chi_hyp}) is verified by using the fact that for each fixed $(x,t)$, $|G^{\lam(x,t)}| \leq r^{-\del}$ by Proposition \ref{prop_K_flat}, so that uniformly in $j,k$
\[ \sup_{(x,t)} \|\chi( (x,t),(u,\theta)) \|_{L^1(dud\theta)} \ll r^{-\del}.\]
Then applying Lemma \ref{lemma_average} with $\lam=r^{-\del}$ and $p=2$ shows that the $L^2(dx,dt)$ norm of the operator (\ref{avg_op_2}) is $\ll r^{-\del/2}2^{-2(j-k)}$, so that the $L^2$ norm of the contribution of $\mathbf{II}$ to (\ref{TTK}) is $\ll r^{-\del/2}2^{-2(j-k)}$. This completes the proof of (\ref{IJ_cases_IJ}), with $\del_0 = \del/4$ and $\ep_0 =1$.

\subsection{Proof of (\ref{IJ_cases_I})}\label{sec_I_I_only}
The last step to complete the proof of Theorem \ref{thm_K_to_S} is to verify the $L^2$ bound in (\ref{IJ_cases_I}) in the case $j<k$. It suffices to prove two bounds: first, a bound with decay in $r$,
\beq\label{I_decay_r}
\| \sup_{\| \lam \| \approx r} |I^\lam_{2^k} \Del_j F|\,\|_{L^2(\R^{n+1})} \ll r^{-\del_0} \|F\|_{L^2(\R^{n+1})}, \qquad j<k, \qquad \text{some $\del_0>0$,}
\eeq
and second, a bound with acceptable growth in $r$ and a decay factor $2^{j-k}$,
\beq\label{I_decay_jk}
\| \sup_{\| \lam \| \approx r} |I^\lam_{2^k} \Del_j F|\,\|_{L^2(\R^{n+1})} \ll r^{1/2}2^{j-k} \|F\|_{L^2(\R^{n+1})}, \qquad j<k.
\eeq
Taking the geometric mean of these bounds then shows that for any real $0\leq \theta \leq 1$,
\[\| \sup_{\| \lam \| \approx r} |I^\lam_{2^k} \Del_j F|\,\|_{L^2(\R^{n+1})} \ll r^{(1-\theta)/2 - \theta \del_0}2^{(j-k)(1-\theta)} \|F\|_{L^2(\R^{n+1})}.\]
This suffices to show (\ref{IJ_cases_I}), as long as $0<\theta<1$ is chosen sufficiently close to 1 that $(1-\theta)/2 - \theta \del_0<0$; such a choice of course exists, since $\del_0>0$.

We will prove (\ref{I_decay_r}) and (\ref{I_decay_jk}) by a $TT^*$ argument for the operator 
\[T_If(x,t)=I^\lam_{2^k} \Del_jf(x,t),\]
where $\lam(x,t)$ is a measurable stopping-time taking values in $[r,2r)$.
We compute that
\beq\label{TITI_kernel} T_IT_I^*f(x,t) = \int_{\R^{n+1}} f(x-u,t- \theta) \frac{1}{a^{n+2}}K_I^{\lam(x,t),\lam(x-u,t-\theta)}(\frac{u}{a},\frac{\theta}{a^2})dud\theta,
\eeq
with kernel
\[ K_I^{\nu,\mu}(u,\theta) = \int_{\R^{n}} e^{iP_\nu(u+z) - iP_\mu(z)} \eta(u+z)\eta(z)  \uDel_{j-k}(\theta-Q(u+z)+Q(z)) dz .\]
We again employ the partition of unity (\ref{W_partition}), writing 
\[ K_I^{\nu,\mu}(u,\theta) = \sum_{l=1}^n W_l(\frac{u}{|u|}) K_{I}^{\nu,\mu}(u,\theta) =:\sum_{l=1}^n K_{I,l}^{\nu,\mu}(u,\theta).\]
 For each $l$ we perform the $l$-th change of coordinates (\ref{subdef}), so that we can then write
 \beq\label{KI_l}
 K_{I,l}^{\nu,\mu}(u,\theta) = 
    \left(\frac{|u_l|}{|u|}\right)^{n-2} W_l\left(\frac{u}{|u|}\right) \int_{\R}  K_{\sharp,l}^{\nu,\mu}(u,\tau)\uDel_{j-k}(\theta-Q(u) - 2|u| \tau) d\tau,
    \eeq
    where for $W_l(u/|u|) \neq 0$,
\beq\label{Ksharp_app2} K_{\sharp,l}^{\nu,\mu}(u,\tau)
    = \int_{\R^{n-1}} e^{i P_\nu(u+z) - iP_\mu(z)} \eta(u+z)\eta(z)d\sig,
    \eeq
    in which $z$ is implicitly defined in terms of $u,\tau,\sig$ as in (\ref{zformula}). Here, $\nu=\lam(x,t)$ and $\mu=\lam(x-u,t-\theta)$. For $W_l( \frac{u}{|u|})=0$, we set $K_{\sharp,l}^{\nu,\mu}(u,\tau)=0$.
We note for later reference that  $K_{\sharp,l}^{\nu,\mu}(u,\tau)$ is supported on $u \in B_2(\R^n), \tau \in B_1(\R)$ (and the integral restricts to $|\sig| \leq c_0^{-1}$ as explained in (\ref{sig_bound})), so that a trivial bound is 
\beq\label{KI_sharp_trivial}
|K_{\sharp,l}^{\nu,\mu}(u,\tau)| \leq c_0^{-(n-1)}\chi_{B_2}(u) \chi_{B_1}(\tau).
\eeq
Moreover, by the hypothesis of Theorem \ref{thm_K_to_S}, a far stronger bound of the form (\ref{K_sharp_theorem_ineq_prequel}) holds for   $ K_{\sharp,l}^{\nu,\mu}$. We apply this strong bound to see that the contribution of the $l$-th term to $T_IT_I^*f$ in (\ref{TITI_kernel}) is bounded by
\beq\label{TITI_part1} \ll \iint_{\R^{n+1}} |f(x-u,t-\theta)|\chi((x,t),(u,\theta))dud\theta,
\eeq
where $\chi((x,t),(u,\theta))$ is the function 
    \begin{multline*}
     \int_{\R} \frac{1}{a^n} \left\{r^{-\del} \chi_{B_2}(\frac{u}{a}) \chi_{B_1}(\tau)+  \chi_{G^{\lam(x,t)}}(\frac{u}{a}) \chi_{B_1(\tau)} +  \chi_{B_2}(\frac{u}{a})\chi_{F_{(\frac{u}{a})}^{\lam(x,t)}}(\tau)\right\}\cdot \\
     \frac{1}{a^2} \uDel_{j-k} \left( \frac{\theta - Q(u) - 2a|u| \tau}{a^2}\right) d\tau.
     \end{multline*}
     By Lemma \ref{lemma_uDel} (i), $\uDel_{j-k}(\theta)$ has bounded $L^1(d\theta)$ norm, uniformly in $j,k$. We may thus apply Lemma \ref{lemma_average}, using the small measures of the sets $G^{\lam(x,t)}$ and $F_{(\frac{u}{a})}^{\lam(x,t)}$ and arguing as we did in (\ref{chi_L1_check}). We  conclude that the $L^2$ norm of the operator (\ref{TITI_part1}) is $\ll r^{-\del/2}$, which suffices to prove (\ref{I_decay_r}) with $\del_0 = \del/4$.
     
To prove the last piece of the puzzle, namely (\ref{I_decay_jk}), we again work with $T_IT_I^*$, since this allows us to isolate an integral in $\tau$, and then use properties of antiderivatives of $\uDel_j$ from Lemma \ref{lemma_uDel} (ii). (See \cite[\S 7.4.1 and \S 10.2]{PieYun19} for an in-depth motivation for this approach.)
Precisely, we return to examine the $l$-th part of the kernel of $T_I T_I^*$ defined in (\ref{KI_l}). By the antidervative property (ii) of Lemma \ref{lemma_uDel}, 
\[ \uDel_{j-k}(\theta - 2|u| \tau - Q(u)) = -\frac{2^{2(j-k)}}{2|u|} \frac{d}{d\tau}\tilde{\uDel}_{j-k}(\theta - 2|u| \tau - Q(u)).\]
Thus after integration by parts in $\tau$,
\beq\label{K_deriv_tau}
K_{I,l}^{\nu,\mu}(u,\theta) = \frac{2^{2(j-k)}}{2|u|}
    \left(\frac{|u_l|}{|u|}\right)^{n-2} W_l\left(\frac{u}{|u|}\right) \int_{\R}  \partial_\tau K_{\sharp,l}^{\nu,\mu}(u,\tau)\tilde{\uDel}_{j-k}(\theta- 2|u| \tau -Q(u) ) d\tau.
    \eeq
Recall the trivial bound (\ref{KI_sharp_trivial}) for $K_{\sharp,l}^{\nu,\mu}(u,\tau)$. In fact, from the definition (\ref{Ksharp_app2}), $K_{\sharp,l}^{\nu,\mu}(u,\tau)$ is a smooth function of $\tau$, upon recalling the definition of $z$ in terms of $u,\tau,\sig$ in (\ref{zformula}). After taking a derivative with respect to $\tau$, we get a  trivial upper bound for $\partial_\tau K_{\sharp,l}^{\nu,\mu}(u,\tau)$ similar to (\ref{KI_sharp_trivial})   but now with a factor of $r$, which appears because the coefficients of the phase $P_\nu(u+z) - P_\mu(z)$ (as a function of $\tau$) are of size $\|\nu\|,\|\mu\| \approx r$. Precisely, we see that
\[ |\partial_\tau K_{\sharp,l}^{\nu,\mu}(u,\tau)| \ll r \cdot  c_0^{-(n-1)}\chi_{B_2}(u) \chi_{B_1}(\tau),
\]
in which the implied constant depends on the fixed polynomials $\{p_1,\ldots, p_M\}$.
Applying this in (\ref{K_deriv_tau}), it follows that uniformly in $l$,
\[  |K_{I,l}^{\nu,\mu}(u,\theta)|  \ll r \frac{2^{2(j-k)}}{2|u|}\chi_{B_2}(u) \int_{\R}  \chi_{B_1}(\tau) \tilde{\uDel}_{j-k}(\theta-Q(u) - 2|u| \tau) d\tau.
    \]
Now we compute the contribution of this term to $T_IT_I^*f$ in (\ref{TITI_kernel}): this portion of the operator is bounded by 
\[\ll r 2^{2(j-k)} \iint_{\R^{n+1}} |f(x-u,t-\theta)|\chi((x,t),(u,\theta))dud\theta,\]
now with the function  
\[
\chi((x,t),(u,\theta))= \int_\R \left( \frac{2|u|}{a}\right)^{-1} \frac{1}{a^n}\chi_{B_2}\left(\frac{u}{a}\right)\chi_{B_1}(\tau)\frac{1}{a^2} \tilde{\uDel}_{j-k} \left( \frac{\theta}{a^2} - 2a|u|\frac{\tau}{a^2} - \frac{Q(u)}{a^2}\right)d\tau.\]
(Note in this case that the function $\chi$ defined above is independent of $(x,t)$ since no stopping-times are present.)
To apply Lemma \ref{lemma_average} we must check what values of $C$ and $\lam$ are valid in the hypothesis (\ref{chi_hyp}) of the lemma. 
We recall that $a=2^k$ is fixed, and compute the integral in $\theta$ first, to see that
\[ \| \sup_{(x,t)}|\chi((x,t),(u,\theta))|\;\|_{L^1(dud\theta)}
    \leq \int_{\R^{n}} 
 \left( \frac{2|u|}{a}\right)^{-1} \frac{1}{a^n}\chi_{B_2}\left(\frac{u}{a}\right) du \cdot
 \int_\R
 \chi_{B_1}(\tau)d\tau \cdot \|\tilde{\uDel}_{j-k}\|_{L^1}.\]
 We recall from Lemma \ref{lemma_uDel} (ii) that the $L^1$ norm of $\tilde{\uDel}_{j-k}$ is bounded uniformly in $j,k$. Then, since $|u|^{-1}$ is locally integrable in $\R^n$ for $n \geq 2$, we see that the integral over $u$ is bounded, uniformly in $a=2^k$. Hence we may apply Lemma \ref{lemma_average} for $p=2$ with finite (but not ``small'') constants $C$ and $\lam=C'$. (Here we do not expect to benefit from any decay in $r$, since no ``small sets'' are present.) We conclude that the contribution of $K^{\nu,\mu}_{I,l}$ to the operator $T_IT_I^*f$ in (\ref{TITI_kernel}) has $L^2$ operator norm bounded by $\ll r 2^{2(j-k)}$. Summing over $1 \leq l \leq n$, we see that in total $T_IT_I^*$ has $L^2$ operator norm $\ll r 2^{2(j-k)}$, and consequently (\ref{I_decay_jk}) holds. 
This completes the proof Theorem \ref{thm_K_to_S}.

To finish the proof of our main result, Theorem \ref{thm_main_R}, we next verify that the key hypothesis of Theorem \ref{thm_K_to_S} holds; that is, that $K_{\sharp,l}^{\nu,\mu}$ admits the appropriate upper bound (\ref{K_sharp_theorem_ineq_prequel}). This is our focus in the next section.

\section{Step 5: Bounding the oscillatory integral $K^{\nu,\mu}_{\sharp,l}$}\label{sec_K}
In this section, we prove the oscillatory integral bound (\ref{K_sharp_theorem_ineq_prequel}) required in Theorem \ref{thm_K_to_S}. We specify that $\ga(y) = (y,Q(y))$ for $Q(y) = \sum_{1 \leq j \leq n} \theta_j y_j^2$, $\theta_j \in \{ \pm 1\}$. At this point, we also restrict our attention to an appropriate class of polynomials for the  $\Span\{p_1,\ldots,p_M\}$, as specified in Theorem \ref{thm_main_R}. Our main result for the oscillatory integral kernel is as follows:
\begin{thm}\label{thm_main_K}
Let $\Psi(u,z)$ be a $C^1$ function supported on $B_2 \times B_1 \subset \R^n \times \R^n$, with 
$\|\Psi\|_{C^1(\sig)} \leq 1$ as in (\ref{Psi_norm}).
Fix $d \geq 2$ and fix nonzero polynomials $p_2,\ldots, p_d$ such that each $p_j$ is homogeneous of degree $j$, and $p_2(y) \neq CQ(y)$ for any $C \neq 0$. 
For each $\lam \in \R^{d-1}$ set 
\[ P_\lam(y) = \sum_{j=2}^d \lam_j p_j(y).\]
For each $1 \leq l \leq n$, for $u\in \R^n$ such that $u/|u|$ lies in the support of the partition $W_l$ defined in (\ref{W_partition}), consider the kernel $K_{\sharp,l}^{\nu,\mu}(u,\tau)$ as defined in (\ref{K_sharp_dfn_Psi}).
Then there exists a small $\del>0$ such that for all $r \geq 1$, if   $\nu,\mu \in \R^{d-1}$ satisfy 
\[ r \leq \|\nu\|,\|\mu\| \leq 2r,\]
then there exists a measurable set $G^\nu \subset B_2(\R^n)$ (depending on $\nu$ but  independent of $\mu,r,\Psi$), and for each $u \in B_2(\R^n)$ there exists a measurable set $F_u^\nu \subset B_1(\R)$ (depending on $u$ and $\nu$ but independent of $\mu, r, \Psi$), such that 
\beq\label{K_small_sets}
|G^\nu| \ll r^{-\del}, \qquad |F_u^\nu| \ll r^{-\del},
\eeq
and 
\beq\label{K_sharp_theorem_ineq}
|K_{\sharp,l}^{\nu,\mu}(u,\tau)| \ll r^{-\del} \chi_{B_2}(u) \chi_{B_1}(\tau) + \chi_{G^\nu}(u) \chi_{B_1}(\tau) + \chi_{B_2}(u)\chi_{F_u^\nu}(\tau),
\eeq
in which all implicit constants depend only on $n$ and the fixed polynomials $p_2,\ldots, p_d$. 
\end{thm}

\begin{remark} Fix any polynomials $p_2,\ldots, p_d$ as in Theorem \ref{thm_main_R}, possibly with some $p_j \con 0$. The supremum in Theorem \ref{thm_main_R} over $P \in \mathcal{Q}_d = \mathrm{span}_\R\{p_2,\ldots,p_d\}$ can only increase if the span of $\mathcal{Q}_d$ is enlarged by including, for each $j$ such that $p_j \con 0$, a homogeneous polynomial of degree $j$  (and not a multiple of $Q$ if $j=2$). Thus it suffices to prove Theorem \ref{thm_main_R} in the case where all $p_2,\ldots,p_d$ are nonzero; this is the case we consider for the remainder of the paper. For such fixed polynomials, Theorem \ref{thm_main_K} proves that the hypothesis of Theorem \ref{thm_K_to_S} holds, so the hypothesis of Theorem \ref{thm_S_to_IJ} holds, so the hypothesis of Theorem \ref{thm_I_to_R} holds, and finally Theorem \ref{thm_main_R} holds.   
\end{remark}

To begin the proof of the theorem, we now fix polynomials $p_j$ homogeneous of degree $j$ as above. 
Fix $1 \leq l \leq n$, and perform the $l$-th change of variables from \S \ref{sec_change_var}, so that $z$ is defined implicitly in terms of $u,\tau,\sig$ as in (\ref{zformula}), for $u/|u|$ in the support of $W_l$. (In all that follows, we only consider such $u$, since $K^{\nu,\mu}_{\sharp,l}(u,\tau)\con0$ for $u/|u|$ outside the support of $W_l$.) Without loss of generality, we consider from now on the case $l=n$ for notational simplicity.

We recall the definition of $K^{\nu,\mu}_{\sharp,n}(u,\tau)$ from (\ref{K_sharp_dfn_Psi}), as an oscillatory integral of compact support in $\sig \in \R^{n-1}$, with phase $P_\nu(u+z)-P_\mu(z)$, which is a polynomial in $\sig$. We   set a notation for the coefficient   of each monomial $\sig^\ga$ in this polynomial, by defining
\[ P_\nu(u+z)-P_\mu(z)=\sum_{0\le |\ga| \le d}C[\sigma^\ga](u,\tau)\sigma^\ga,\]
where $\ga$ varies over multi-indices $\ga \in \Z^{n-1}_{\geq 0}$.
We define a norm for the coefficients of non-constant terms in this polynomial in $\sig$:
\[ \|P_\nu(u+z)-P_\mu(z)\|_\sigma:=\sum_{1\le |\ga|\le d}|C[\sigma^\ga](u,\tau)|. \]
Fix $\ep_1>0$. By a standard van der Corput estimate, which we recall below in Lemma \ref{lemma_VdC_int},
for each $(u,\tau)$ such that   
\beq\label{P_coeff_big}
\|P_\nu(u+z)-P_\mu(z)\|_\sigma \gg r^{\ep_1},
\eeq
then $|K^{\nu,\mu}_{\sharp,n}(u,\tau)| \ll r^{-\del}$ for $\del=\ep_1/d$. We recall the estimate:

\begin{lemma}\label{lemma_VdC_int}[Prop. 2.1 of \cite{SWCarl}]
Let $Q_\lam(x) = \sum_{0 \leq |\al| \leq d} \lam_\al x^\al$ be a real-valued polynomial for $x \in \R^m$. Define $\| \lam \| = \sum_{1 \leq |\al|\leq d} |\lam_\al|$ (omitting the constant term in $Q_\lam$). For any $C^1$ function $\psi$ defined on $B_1(\R^m)$ such that $\|\psi\|_{C^1} \leq 1$, and for any convex subset $\Omega \subseteq B_1(\R^m)$, 
\[ |\int_\Omega e^{i Q_\lam(x)} \psi(x) dx| \leq_{m,d} \|\lam\|^{-1/d},\]
where the implicit constant is independent of $\psi, \Omega$.
\end{lemma}
Our main goal is to prove that when $\|\nu\|,\|\mu\| \approx r$, for ``most'' $(u,\tau)$, $\|P_\nu(u+z)-P_\mu(z)\|_\sigma \gg r^{\ep_1}$ so that $|K^{\nu,\mu}_{\sharp,n}(u,\tau)| \ll r^{-\del}$.
However, for certain $(u,\tau) \in \R^{n+1}$ we cannot prove that $\|P_\nu(u+z)-P_\mu(z)\|_\sigma \gg r^{\ep_1}$; it could be that for each multi-index $\ga$ the coefficient $C[\sig^\ga](u,\tau)$ is small. But these coefficients contain expressions that are polynomial in $u/|u|$ and $\tau$, so our strategy is to show that the set of such $(u,\tau)$  has small measure. This will follow from a modification of a van der Corput estimate, which we will apply to coefficients of $u$ and $\tau$ inside $C[\sig^\ga](u,\tau)$:

\begin{lemma}\label{lemma_VdC_set}[Lemma 3.3 of \cite{PieYun19}]
Let $Q_\lam(x) = \sum_{0 \leq |\al| \leq d} \lam_\al x^\al$ be a real-valued polynomial for $x \in \R^m$. Define $\llbracket \lam \rrbracket = \sum_{1 \leq |\al|\leq d} |\lam_\al| + |Q_\lam(0)|$ (including the constant term in $Q_\lam$).
For every $\rho>0$,
\[ |\{x \in B_1(\R^m) : |Q_\lam(x)| \leq \rho\}| \ll_{m,d} \rho^{1/d}\llbracket \lam \rrbracket^{-1/d} .\]
\end{lemma}

Furthermore, it is significant that the sets claimed to exist in (\ref{K_small_sets}) and (\ref{K_sharp_theorem_ineq}) are completely independent of $\mu$, although they are allowed to depend on $\nu$; recall Remark \ref{remark_mu}.  Thus a key point of the strategy we now employ to prove Theorem \ref{thm_main_K} is to avoid any dependence on the ``bad'' stopping-time $\mu$.

We  begin our study of $P_\nu(u+z) - P_\mu(z)$, aiming to prove (\ref{P_coeff_big}) whenever possible. 
  Fix a dimension $n \geq 2$. Recall we had a distinguished coordinate $l = n$. 
 Given a multi-index $\al=(\al_1,\ldots,\al_{n-1}) \in \mathbb{Z}_{\geq 0}^{n-1}$, 
 an index $\lambda \in \mathbb{Z}_{\geq 0}$, and given $u \in \R^{n}$, we write 
 $u^{(\al;\lambda)} = u_1^{\al_1}\cdots u_{n-1}^{\al_{n-1}} u_n^{\lambda}$. Similarly,  $\partial_{(\al;\lambda)}$ indicates taking $\al_1$ derivatives in the first coordinate, $\al_2$ derivatives in the second coordinate, and finally $\lambda$ derivatives in the $n$-th coordinate.

Recall also the notation from \S \ref{sec_change_var}, that
 \[ \tilde{u} := (\theta_1u_1,\ldots, \theta_nu_n), \]
where $\theta_i \in \{\pm 1\}$ are the eigenvalues of the fixed quadratic form $Q$. 
Write $\tu = (\tu', \tu_n)$ and $z = (z',z_n)$ with\[
z' = \frac{\tau \tu' - \tu_n \sigma}{|u|}, \quad z_n = \frac{\tau \tu_n + \tu' \cdot \sigma}{|u|}.
\]
Now $P_{\nu}(u+z) - P_{\mu}(z)$ is a polynomial in $\sigma \in \R^{n-1}$. The coefficient $C[\sigma^{\gamma}](u,\tau)$ of this polynomial can then be computed by
\begin{equation} \label{eq:C}
	C[\sigma^{\gamma}](u,\tau) 
	= \left. \frac{1}{\gamma!} \partial^{\gamma}_{\sigma} \Big[P_{\nu}(u+z) - P_{\mu}(z) \Big] \right|_{\sigma = 0}
\end{equation}
where $\gamma$ will always be a multi-index in $\Z_{\geq 0}^{n-1}$. Using Taylor expansion and the chain rule,  
\begin{align}
	\left. \frac{1}{\gamma!} \partial^{\gamma}_{\sigma} P_{\mu}(z) \right|_{\sigma = 0}
	&= \sum_{\alpha \leq \gamma} \frac{1}{\alpha!(\gamma-\alpha)!} [\partial_{(\alpha;|\gamma-\alpha|)} P_{\mu}]\Big(\frac{\tau \tu}{|u|}\Big) \Big( \frac{\tu}{|u|} \Big)^{\gamma-\alpha} \Big( \frac{-\tu_n}{|u|} \Big)^{|\alpha|} \nonumber \\
	&= \sum_{j=2}^d \mu_j \sum_{\alpha \leq \gamma} \frac{(-1)^{|\alpha|}}{\alpha!(\gamma-\alpha)!} [\partial_{(\alpha;|\gamma-\alpha|)} p_j]\Big(\frac{\tau \tu}{|u|}\Big) \frac{\tu^{(\gamma-\alpha;|\alpha|)}}{|u|^{|\gamma|}}\nonumber \\
	&= \sum_{j=2}^d \mu_j \sum_{\alpha \leq \gamma} \sum_{|\beta| \leq j-|\gamma|} \frac{(-1)^{|\alpha|}}{\alpha!(\gamma-\alpha)!\beta!} [\partial_{(\alpha;|\gamma-\alpha|)+\beta} p_j](0) \frac{\tu^{(\gamma-\alpha;|\alpha|)+\beta}}{|u|^{|\gamma|+|\beta|}} \tau^{|\beta|}. \label{eq:Pmu0}
\end{align}
Henceforth $\alpha$ and $\beta$ will always be  multi-indices in $\Z_{\geq 0}^{n-1}$ and $\Z_{\geq 0}^n$, respectively. In fact, since $p_j$ is homogeneous of degree $j$, we only need to sum over $\beta$ with $|\beta| = j-|\gamma|$, and for such $\beta$, $[\partial_{(\alpha;|\gamma-\alpha|)+\beta} p_j](0) = [\partial_{(\alpha;|\gamma-\alpha|)+\beta} p_j](u)$. (This evaluation at $u$ will better match (\ref{eq:Pnu}) below, leading to a simplified presentation of  (\ref{eq:C_expansion}).) Thus
\begin{equation} \label{eq:Pmu}
\left. \frac{1}{\gamma!} \partial^{\gamma}_{\sigma} P_{\mu}(z) \right|_{\sigma = 0} 
= \sum_{j=2}^d \mu_j \sum_{\alpha \leq \gamma} \sum_{|\beta| = j-|\gamma|} \frac{(-1)^{|\alpha|}}{\alpha!(\gamma-\alpha)!\beta!} [\partial_{(\alpha;|\gamma-\alpha|)+\beta} p_j](u) \frac{\tu^{(\gamma-\alpha;|\alpha|)+\beta}}{|u|^{|\gamma|+|\beta|}} \tau^{|\beta|}.
\end{equation}
Similarly to \eqref{eq:Pmu0},
\begin{align} 
	\left. \frac{1}{\gamma!} \partial^{\gamma}_{\sigma} P_{\nu}(u+z) \right|_{\sigma = 0}
	&= \sum_{\alpha \leq \gamma} \frac{1}{\alpha!(\gamma-\alpha)!} [\partial_{(\alpha;|\gamma-\alpha|)} P_{\nu}]\Big(u+\frac{\tau \tu}{|u|}\Big) \Big( \frac{\tu}{|u|} \Big)^{\gamma-\alpha} \Big( \frac{-\tu_n}{|u|} \Big)^{|\alpha|} \nonumber \\
	&= \sum_{j=2}^d \nu_j \sum_{\alpha \leq \gamma} \frac{(-1)^{|\alpha|}}{\alpha!(\gamma-\alpha)!} [\partial_{(\alpha;|\gamma-\alpha|)} p_j]\Big(u+\frac{\tau \tu}{|u|}\Big) \frac{\tu^{(\gamma-\alpha;|\alpha|)}}{|u|^{|\gamma|}} \nonumber \\
	&= \sum_{j=2}^d \nu_j \sum_{\alpha \leq \gamma} \sum_{|\beta| \leq j-|\gamma|} \frac{(-1)^{|\alpha|}}{\alpha!(\gamma-\alpha)!\beta!} [\partial_{(\alpha;|\gamma-\alpha|)+\beta} p_j](u) \frac{\tu^{(\gamma-\alpha;|\alpha|)+\beta}}{|u|^{|\gamma|+|\beta|}} \tau^{|\beta|}. \label{eq:Pnu}
\end{align}
It will be convenient to introduce the notation
\[
R_{p,\gamma,b}(u) := \sum_{\alpha \leq \gamma} \sum_{|\beta| = b} \frac{(-1)^{|\alpha|}}{\alpha!(\gamma-\alpha)!\beta!} [\partial_{(\alpha;|\gamma-\alpha|)+\beta} p](\tu) u^{(\gamma-\alpha;|\alpha|)+\beta}
\]
for any polynomial $p$, any multi-index $\gamma \in \Z_{\geq 0}^{n-1}$ and any $b \in \Z_{\geq 0}$. In particular, $R_{p_j,\gamma,b}(u)$ is homogeneous of degree $j$ for all $j, \gamma, b$.   The relations \eqref{eq:C}, \eqref{eq:Pmu} and \eqref{eq:Pnu} imply that
\begin{equation} \label{eq:C_expansion}
C[\sigma^{\gamma}](u,\tau)
= \sum_{j=2}^d (\nu_j-\mu_j) R_{p_j,\gamma,j-|\gamma|}\Big(\frac{\tu}{|u|}\Big) \tau^{j-|\gamma|} +  \sum_{j=2}^d \sum_{0 \leq b < j-|\gamma|} \nu_j  R_{p_j,\gamma,b}(\tu) \frac{\tau^b}{|u|^{|\gamma|+b}}.
\end{equation}
Here, it is implicit that the first sum is only over $j \geq |\ga|$.

Our next aim is to show at least one such coefficient $C[\sigma^{\gamma}](u,\tau)$ is large. We now package the system of such coefficients by defining appropriate vectors, so that the system of relations (\ref{eq:C_expansion}), which depends linearly on the $\nu_j$ and $\mu_j$, can be expressed as matrix multiplication. Write $\nu = (\nu_j)_{2 \leq j \leq d}$ and $\mu = ( \mu_j)_{2 \leq j \leq d}$ as column vectors. Now define the function $B_{j,\gamma}(u) := R_{p_j,\gamma,j-|\gamma|}(u)$ for each integer $j \geq |\gamma|$.  Let $\Bbf_{j,k}$ be the column vector $\Big(B_{j,\gamma}(\frac{\tu}{|u|}) \Big)_{|\gamma| = k}$ when $j \geq k$, indexed by multi-indices $\ga \in \Z_{\geq 0}^{n-1}$. 
(For example,  for each $j$, $\Bbf_{j,1}$ is a column vector with $n-1$ entries.)   Define the matrix
\[
\bb = \left( 
\begin{array}{ccccc}
	\tau \Bbf_{2,1} & \tau^2 \Bbf_{3,1} & \tau^3 \Bbf_{4,1} & \dots & \tau^{d-1} \Bbf_{d,1} \\
	\Bbf_{2,2} & \tau \Bbf_{3,2} & \tau^2 \Bbf_{4,2} & \dots & \tau^{d-2} \Bbf_{d,2} \\
	0 & \Bbf_{3,3} & \tau \Bbf_{4,3} & \dots & \tau^{d-3} \Bbf_{d,3} \\
	0 & 0 & \Bbf_{4,4} & \dots & \tau^{d-4} \Bbf_{d,4} \\
	0 & 0 & 0 & \ddots & \vdots \\
	0 & 0 & 0 & 0 & \Bbf_{d,d}
\end{array}
\right).
\]
This has $d-1$ columns and $\sum_{1 \leq |\ga| \leq d}1$ rows, and is independent of both $\nu,\mu$. Moreover, by construction, each entry in $\bb$ is bounded above by a constant depending only on $n,d$ and the fixed polynomials $p_2,\ldots,p_d$, uniformly for $u,\tau$ in their compact supports (the balls $B_2$ and $B_1$, respectively).  Write
$\bF$ for the column vector 
\[
\bF:= \Big( \sum_{j=2}^d \sum_{0 \leq b < j-|\gamma|} \nu_j  R_{p_j,\gamma,b}(\tu) \frac{\tau^b}{|u|^{|\gamma|+b}} \Big)_{1 \leq |\gamma| \leq d}.
\]
 Note $\bF$ depends only on $\nu$ but not on $\mu$.
Finally, write $\bc$ for the column vector $(C[\sigma^{\gamma}](u,\tau))_{1 \leq |\gamma| \leq d}$. Then as the multi-index $\ga \in \Z_{\geq 0}^{n-1}$ varies, all the expansions \eqref{eq:C_expansion} can be encapsulated in the vector identity
\begin{equation} \label{eq:key}
\bc = \bb (\nu-\mu) + \bF.
\end{equation}

Now it is crucial to avoid any dependence on $\mu$ when we estimate the size of any entry in the column vector $\bc$. One way to accomplish this is to eliminate $\bb(\nu-\mu)$ on the right hand side, and this can be done if we apply a linear operator on both sides that annihilates the image (column space) of $\bb$. If $\bb^T \bb$ is invertible, then one good choice is the orthogonal projection onto the nullspace of $\bb^T$, given by the projection matrix $I - \bb (\bb^T \bb)^{-1} \bb^T$ (note that the nullspace of $\bb^T$ is non-trivial since $\bb$ has more rows than columns). But $\bb^T \bb$ may not be invertible, and in any case it will be convenient not to have to divide by the determinant of $\bb^T \bb$ when we compute the inverse to $\bb^T \bb$. Thus we multiply \eqref{eq:key} by $\det(\bb^T \bb) I - \bb \adj(\bb^T \bb) \bb^T$ instead. (Here $\adj(\cdot)$ denotes the adjugate matrix.) Since $ [ \det(\bb^T \bb) I - \bb \adj(\bb^T \bb) \bb^T ] \bb = 0$, the vector identity (\ref{eq:key}) implies
\begin{equation} \label{eq:key_proj}
	(\det(\bb^T \bb) I - \bb \adj(\bb^T \bb) \bb^T) \bc = (\det(\bb^T \bb) I - \bb \adj(\bb^T \bb) \bb^T) \bF.
\end{equation}
Next, the idea  is to show that the vector on the right-hand side has a large entry. All entries in $\det(\bb^T \bb) I - \bb \adj(\bb^T \bb) \bb^T$ are bounded above uniformly for all $(u,\tau) \in B_2 \times B_1$, by some constant $C_0$ that depends only on  $n,d, p_2,\ldots,p_d$. Hence,  if an entry in the vector on the right-hand side is $\geq r^{\ep_1}$ for a given $\ep_1>0$, then the vector $\bc$ must have an entry that is $\gg r^{\ep_1}$, with an implicit constant depending only on $C_0$ and the length of the vector (i.e. on $n,d$). 

To study the right-hand side of (\ref{eq:key_proj}), we Taylor expand $\det(\bb^T \bb) I - \bb \adj(\bb^T \bb) \bb^T$ in powers of $\tau$. We will be able to detect large contributions to the right-hand side by studying terms that are constant or linear with respect to $\tau$, and thus we will work modulo terms that are $O(\tau^2)$ from now on. First observe that the top $(n-1)$ rows of this matrix (indexed by those $\gamma$ with $|\gamma| = 1$) are given by 
\beq\label{BB_mod}
\prod_{j=3}^d |\Bbf_{j,j}|^2 \left(
\begin{array}{ccccc}
|\Bbf_{2,2}|^2 I_{n-1} & -\tau \Bbf_{2,1} \Bbf_{2,2}^T & 0 & \dots & 0
\end{array}
\right) + O(\tau^2).
\eeq
Here we denote by $|\Bbf_{j,j}|^2= \Bbf_{j,j}^T \Bbf_{j,j}$  the norm (squared) of the vector $\Bbf_{j,j}$ defined earlier, and $I_{n-1}$ is   the $(n-1)\times(n-1)$ identity matrix.
The above expression holds because first of all,
\[
\det(\bb^T \bb) = \prod_{j=2}^d |\Bbf_{j,j}|^2 + O(\tau^2).
\]
Second, up to $O(\tau)$ terms, $\adj(\bb^T \bb)$ is a $(d-1)\times(d-1)$ dimensional diagonal matrix whose $(i,i)$-th entry is $\prod_{\substack{2 \leq j \leq d \\ j \ne i}} |\Bbf_{j,j}|^2$.
The top $(n-1)$ rows of $\bb$ are given by 
\[
\left(
\begin{array}{ccccc}
	\tau \Bbf_{2,1} & 0 & 0 & \dots & 0
\end{array}
\right) + O(\tau^2),
\]
so that all nonzero terms contain a factor of $\tau$. Thus to express the top $n-1$ rows of $\bb \adj(\bb^T \bb) \bb^T$ modulo $O(\tau^2)$ terms, it suffices to consider  $\bb^T$ modulo $O(\tau)$ (which has very few nonzero entries). 
This gives the Taylor expansion of the first $(n-1)$ rows of $\bb \adj(\bb^T \bb) \bb^T$ modulo $O(\tau^2)$ as
\[
\prod_{j=3}^d |\Bbf_{j,j}|^2 \left(
\begin{array}{ccccc}
0_{n-1} & -\tau \Bbf_{2,1} \Bbf_{2,2}^T & 0 & \dots & 0
\end{array}
\right) + O(\tau^2),
\]
in which $0_{n-1}$ denotes an $(n-1)\times (n-1)$  matrix of zeroes.
Assembling these facts, the top $(n-1)$ rows of $\det(\bb^T \bb) I - \bb \adj(\bb^T \bb) \bb^T$ modulo $O(\tau^2)$ are as claimed in (\ref{BB_mod}). 

We may enumerate the $n-1$ multi-indices $\ga$ with $|\ga|=1$ by $e_1,\ldots, e_m$ for $1 \leq m \leq n-1$ where $e_m$ denotes the standard unit vector in $\Z^{n-1}$. 
It follows that for $1 \leq m \leq n-1$, the $m$-th entry in the column vector that comprises the right hand side of \eqref{eq:key_proj} is 
\begin{align*}
&\prod_{j=2}^d |\Bbf_{j,j}|^2 \sum_{j=2}^d \nu_j  R_{p_j,e_m,0}(\tu) \frac{1}{|u|} \\
&+ \tau \prod_{j=3}^d |\Bbf_{j,j}|^2 \left(|\Bbf_{2,2}|^2 \sum_{j=2}^d \nu_j R_{p_j,e_m,1}(\tu) \frac{1}{|u|^2}- B_{2,e_m}(\frac{\tu}{|u|}) \Bbf_{2,2}^T \sum_{j=2}^d \nu_j \mathbf{R}_{p_j,2,0} \frac{1}{|u|^2} \right) + O(\tau^2)
\end{align*}
where $\mathbf{R}_{p_j,2,0} := (R_{p_j,\gamma,0}(\tu))_{|\gamma|=2}$ is a vector. Using homogeneity, we rewrite this as
\begin{align}
	\frac{1}{|u|^{d(d+1)-1}}& \left(  \prod_{j=2}^d \sum_{|\gamma|=j} |B_{j,\gamma}(\tu)|^2 \right) \sum_{j=2}^d \nu_j  R_{p_j,e_m,0}(\tu) \label{R_expression}\\
	&+ \frac{\tau}{|u|^{d(d+1)}} \left(  \prod_{j=3}^d \sum_{|\gamma|=j} |B_{j,\gamma}(\tu)|^2 \right) \nonumber \\
   & \cdot \sum_{j=2}^d \nu_j \left( \sum_{|\gamma|=2} |B_{2,\gamma}(\tu)|^2 R_{p_j,e_m,1}(\tu) - \sum_{|\gamma|=2} B_{2,\gamma}(\tu) R_{p_j,\gamma,0}(\tu) B_{2,e_m}(\tu) \right)\nonumber \\
 & + O(\tau^2). \nonumber
\end{align}
To finish the proof of Theorem \ref{thm_main_K}, we aim to show that if $\|\nu\|  \approx r$, then there exists an index $m$ with $1 \leq m \leq n-1$ such that the above expression is $\gg r^{\ep_1}$ (for a fixed $\ep_1>0$) for all but a small exceptional set of $u,\tau$.  Henceforward, we will say that   $p_j$ is $Q$-type if there exists a constant $C$ such that $p_j(y) =C Q(y)^{j/2};$  necessarily $j$ is then even.
We  will employ the following three claims. 
\begin{lemma} \label{lem:1}
	For a given $2 \leq j \leq d$, if $p_j \not\equiv 0$, then $\sum_{|\gamma|=j} |B_{j,\gamma}(u)|^2  \not\equiv 0$.
\end{lemma}
\begin{lemma} \label{lem:2}
For a given   $2 \leq j \leq d$, if $p_j\not\con 0$ is not $Q$-type then there exists $1 \leq m \leq n-1$ so that $R_{p_j,e_m,0}(u) \not\equiv 0$.
\end{lemma}
\begin{lemma} \label{lem:nonvanishing}
For a given (even)  $4 \leq j \leq d$, if $p_j(y) \not\equiv 0$ is $Q$-type, and $p_2\not\con 0$ is not $Q$-type, then there exists $1 \leq m \leq n-1$ such that
\begin{equation} \label{eq:lem3}
\sum_{|\gamma|=2} |B_{2,\gamma}(u)|^2 R_{p_j,e_m,1}(u) - \sum_{|\gamma|=2} B_{2,\gamma}(u) R_{p_j,\gamma,0}(u) B_{2,e_m}(u)  \not\equiv 0.
\end{equation}
\end{lemma}

\subsection{Completing the proof of Theorem \ref{thm_main_K}}
The proofs of the lemmas will be given later; for now we assume them and finish the proof of Theorem \ref{thm_main_K}. Given that $r \leq \|\nu\|  \leq 2r$, there is at least one index $j$ for which $|\nu_j| \geq c r$ for a fixed small constant  $c$ ($c =1/(d-1)$ will do).  Given $\nu$, there are two cases:  there is an index $j$ for which $|\nu_j| \geq c r$ and $p_j$ is not $Q$-type (case A), or for all indices $j$ with $|\nu_j | \geq c r$, $p_j$ is $Q$-type (case B).  Note that under the hypothesis of the theorem that $p_2\not\con 0$ is not $Q$-type, if $d=2$ or $d=3$ then only case A can occur.

We first consider case A, so there exists an index $2 \leq j^*\leq d$ with $|\nu_{j^*}| \geq cr$ and $p_{j^*}$ not $Q$-type, so that by Lemma \ref{lem:2} we may choose an index $1\leq m \leq n-1$ for which $R_{p_{j^*},e_m,0}\not\con 0$. We fix this $m$ for the remainder of the argument for case A.  For this $m$, the expression (\ref{R_expression}), call it $S^m_{\nu,u}(\tau)$, is a polynomial in $\tau$ of degree $D \ll_{n,d} 1$, whose $\llbracket S^m_{\nu,u}\rrbracket_\tau$ norm, in the notation of Lemma \ref{lemma_VdC_set}, is bounded below by the absolute value of the constant term with respect to $\tau$. With this in mind, we define a polynomial in $u$ with degree $E \ll_{d} 1$ by
\[
W^m_\nu(u)=\left(  \prod_{j=2}^d \sum_{|\gamma|=j} |B_{j,\gamma}(\tu)|^2 \right) \sum_{j=2}^d \nu_j  R_{p_j,e_m,0}(\tu).
\]
As a consequence of Lemma \ref{lem:1}, the leading factor is a nonzero polynomial in $u$ (with coefficients independent of $\nu,r$). 

Fix $0<\ep_1< \ep_2<1$. We define a set 
\[ G^\nu = \{ u \in B_2(\R^n): |W^m_\nu(u)| \leq r^{\ep_2}\}.\]
For each $u\not\in G^\nu$, define the set
\[ F^\nu_u = \{\tau \in B_1(\R) : |S^m_{\nu,u}(\tau)| \leq r^{\ep_1}\}.\]
For each $u \in G^\nu$, define $F^\nu_u = \emptyset.$ Now for all $(u,\tau) \in B_2(\R^n) \times B_1(\R)$ such that $u \not\in G^\nu$ and $\tau \not\in F^\nu_u$, then $|S^m_{\nu,u}(\tau)| \geq r^{\ep_1}$, which is equivalent to the statement that the $m$-th entry on the right-hand side of (\ref{eq:key_proj}) is $\geq r^{\ep_1}$ and consequently there exists an entry in the vector $\mathbf{C}$ that is $\gg r^{\ep_1}$. That is to say, there is some multi-index $\ga$ such that $|C[\sig^\ga](u,\tau)| \gg r^{\ep_1}$, and hence by Lemma \ref{lemma_VdC_int}, $|K_{\sharp,n}^{\nu,\mu}(u,\tau)| \ll r^{-\ep_1/d}$. 

By construction, $\llbracket S^m_{\nu,u}\rrbracket_\tau \geq |u|^{-(d(d+1)-1)} |W^m_\nu(u)| \gg |W^m_\nu(u)|$ for $u \in B_2$. Hence if $u\not\in G^\nu$, $\llbracket S^m_{\nu,u}\rrbracket_\tau \gg r^{\ep_2}$ so that by Lemma \ref{lemma_VdC_set}, $|F^\nu_u| \ll r^{-(\ep_2-\ep_1)/D}.$
Finally, the set $G^\nu$ is also small. Indeed,  since each $R_{p_j,e_m,0}$ is homogeneous of  degree $j$ and $m$ has been fixed so that $R_{p_{j^*},e_m,0}\not\con0$,  we can bound $\llbracket W_\nu^m\rrbracket_u \geq |\nu_{j^*}| \gg r$. Consequently, by Lemma \ref{lemma_VdC_set}, $|G^\nu| \ll r^{-(1-\ep_2)/E}.$ This completes the deduction of Theorem \ref{thm_main_K} in case A.

Next consider case $B$, so that for all $j$ with $|\nu_j| \geq cr$, $p_j$ is $Q$-type. As noted, we may assume in this case that $d \geq 4$, and in particular there exists $j^* \geq 4$ such that $|\nu_{j^*}| \geq cr$ and $p_{j^*}$ is $Q$-type. The only change we make in the above argument is that now we define the polynomial 
\begin{multline*}
W_\nu^m(u)= \left(  \prod_{j=3}^d \sum_{|\gamma|=j} |B_{j,\gamma}(\tu)|^2 \right) \cdot \\
\sum_{j=2}^d \nu_j \left( \sum_{|\gamma|=2} |B_{2,\gamma}(\tu)|^2 R_{p_j,e_m,1}(\tu) - \sum_{|\gamma|=2} B_{2,\gamma}(\tu) R_{p_j,\gamma,0}(\tu) B_{2,e_m}(\tu) \right)
\end{multline*}
extracted from the term in $S_{\nu,u}^m$ that is linear in $\tau$; its degree in $u$ we again denote by $E \ll_d 1$. By Lemma \ref{lem:nonvanishing}, we may fix an index $1 \leq m \leq n-1$ for which the factor multiplying $\nu_{j^*}$ in the above expression for $W_\nu^m(u)$ is not the zero polynomial. With this choice for $m$, $\llbracket S^m_{\nu,u}\rrbracket_\tau \gg |W_\nu^m(u)|$, and the argument above proceeds verbatim, concluding the deduction of Theorem \ref{thm_main_K} in case B.

 \begin{remark}\label{remark_no_linear}
 Suppose that a linear polynomial $p_1 \not\con 0$ is included in the fixed set $p_1,p_2,\ldots,p_d$; in this case the sums in the expansion  (\ref{eq:C_expansion}) are nominally indexed from $j=1$. Any argument must allow for the case  that $\|\nu\|,\|\mu\|\approx r$ and $|\nu_1|,|\mu_1| \gg r$ while $|\nu_j|\approx |\mu_j|\approx 0$ for all $j \geq 2$. The only potentially large coefficients are then $C[\sig^\ga](u,\tau)$ with $|\ga|=1$, say $\ga=e_m$ for $1 \leq m \leq n-1$, in which case $C[\sig_m](u,\tau)=(\nu_1- \mu_1)R_{p_1,e_m,0}(\tilde{u}/|u|)$. This   could vanish identically for each $1 \leq m \leq n-1$  when $\nu_1\approx \mu_1$, and there is no remaining term that depends only on $\nu_1$ and can be exploited. Analogously, the projection carried out in (\ref{eq:key_proj}) zeroes out the entire coefficient, and the argument cannot proceed. This illustrates why Theorem \ref{thm_main_K}, and hence the main Theorem \ref{thm_main_R}, prohibits a linear phase contribution.
\end{remark}

 We now finish the proof of Theorem \ref{thm_main_K} by proving the three lemmas.
 
\subsection{Proof of Lemma~\ref{lem:1}}
Recall that by definition $B_{j,\gamma}(u)=R_{p_j,\gamma,j-|\gamma|}(u)$. If $B_{j,\gamma}(u) \equiv 0$ for all $\gamma$ with $|\gamma| = j$, then $R_{p_j,\gamma,0}(u) \equiv 0$ for all such $\gamma$. Each such $R_{p_j,\gamma,0}(u)$ is a sum of monomials in $u$ of distinct multi-index exponents, whose coefficients must all vanish; this  implies $\partial_{(\alpha;j-|\alpha|)}p_j \con 0$ for all $\alpha$ with $|\alpha| \leq j$. Hence $p_j \equiv 0$, and the lemma is proved.

\subsection{Proof of Lemma~\ref{lem:2}}
Suppose to the contrary that
\begin{equation} \label{eq:hypo}
0 \equiv R_{p_j,e_m,0}(u) = [\partial_{e_n} p_j](\tu) u_m - [\partial_{e_m} p_j](\tu) u_n,
\end{equation}
for all $1 \leq m \leq n-1$. Then $u_n$ divides $\partial_{e_n} p_j(u)$ and $u_m$ divides $\partial_{e_m} p_j(u)$ for $1 \leq m \leq n-1$. Thus one may write $\partial_{e_l} p_j(u) = \tu_l q_l(u)$ for some polynomial $q_l(u)$, for each $1 \leq l \leq n$. Plugging this back into \eqref{eq:hypo} gives $q_1(\tu) = \dots = q_n(\tu)$ for all $u$, so $\nabla p_j(u)$ is parallel to the vector $\tu$ at every $u \in \R^n$. In particular, since $\tu$ is a normal vector to the level sets of $Q(u)$, this shows $p_j(u)$ is constant on such level sets. Now the level set $\{u \in \R^n \colon Q(u) = 1\}$ can be written as the disjoint union of finitely many connected components $\Sigma_{+,a}$ indexed by $a$, and similarly the level set $\{u \in \R^n \colon Q(u) = -1\}$ can be written as the disjoint union of finitely many connected components $\Sigma_{-,b}$ indexed by $b$. The set of positive dilates of each of $\Sigma_{+,a}$ and $\Sigma_{-,b}$ generates a cone in $\R^n$, say $\Gamma_{+,a}$ and $\Gamma_{-,b}$ respectively. The space $\R^n$ is the disjoint union of the zero set of $Q$ with $\bigcup_a \Gamma_{+,a}$ and $\bigcup_b \Gamma_{-,b}$. We know for each $b$, there exists a constant $c_{-,b}$ such that $p_j(u) = c_{-,b}$ whenever $u \in \Sigma_{-,b}$. Now let $u \in \Gamma_{-,b}$. Writing temporarily $t = |Q(u)|^{1/2}$ so that $\frac{u}{t} \in \Sigma_{-,b}$, this shows 
\[p_j(u) = t^j p_j(\frac{u}{t}) = c_{-,b} t^j = c_{-,b} |Q(u)|^{j/2} = c_{-,b} (-1)^{j/2} Q(u)^{j/2}.\]
This proves $p_j(u) = c_{-,b} (-1)^{j/2} Q(u)^{j/2}$ for all $u \in \Gamma_{-,b}$. Moreover, since by hypothesis $p_j$ has real coefficients, $j$ must be even. Similarly there are constants so that $p_j(u) = c_{+,a} Q(u)^{j/2}$ for all $u \in \Gamma_{+,a}$. By continuity of the polynomial $p_j$, we must have $c_{-,b}(-1)^{j/2} = c_{+,a}$ for all $a$ and $b$. Thus $p_j(u)$ is a constant multiple of $Q(u)^{j/2}$, namely, $Q$-type. The lemma is proved.
 
\subsection{Proof of Lemma~\ref{lem:nonvanishing}}\label{sec_Ksharp_nonvanishing}
First we claim that for any even index $2 \leq j \leq d$, and for each fixed $1 \leq m \leq n-1$, the left-hand side of \eqref{eq:lem3} factors into precisely the expression
\beq\label{dep_only_2}
\frac{j}{2} Q(\tu)^{\frac{j}{2}-1} \sum_{|\gamma|=2} [|R_{p_2,\gamma,0}(u)|^2 R_{Q,e_m,1}(u) - R_{p_2,\gamma,0}(u) R_{Q,\gamma,0}(u) R_{p_2,e_m,1}(u)].
\eeq
\begin{remark}\label{remark_no_quad}
This has two interesting features. First, the only dependence on $j$ is in the first factor $(j/2)Q(\tilde{u})^{j/2-1}$; whether the expression is identically zero thus hinges  upon the sum over $|\ga|=2$, which depends only on $m$ and the functions $p_2, Q$. Second, (\ref{dep_only_2}) is identically zero (for all $m$) if $p_2$ is $Q$-type. This illustrates why Theorem \ref{thm_main_K}, and hence the main Theorem \ref{thm_main_R}, prohibits $p_2(y) \con C Q(y)$.  
\end{remark}
Once (\ref{dep_only_2}) has been verified, Lemma~\ref{lem:nonvanishing} is an immediate consequence of the following:
\begin{lemma} \label{lem:nonvanishing2}
Let $p_2\not\con 0$ be a homogeneous polynomial of degree $2$ on $\R^n$. If
\begin{equation} \label{eq:nonvanishing}
\sum_{|\gamma|=2} |R_{p_2,\gamma,0}(u)|^2 R_{Q,e_m,1}(u) \equiv \sum_{|\gamma|=2} R_{p_2,\gamma,0}(u) R_{Q,\gamma,0}(u) R_{p_2,e_m,1}(u)
\end{equation}
for all $1 \leq m \leq n-1$, then $p_2(u)$ is a multiple of $Q(u)$.
\end{lemma}
We verify (\ref{dep_only_2}) by directly computing each quantity in (\ref{eq:lem3}). We denote the coefficients of $p_2$ by setting
\begin{equation} \label{eq:p2def}
p_2(u) = \sum_{1 \leq r, s \leq n} c_{r,s} u_{r} u_{s}
\end{equation}
with $c_{r,s} = c_{s,r}$. First we compute terms involving $|\ga|=1$, which we denote by $e_m$ with $1 \leq m \leq n-1$.
For each $1 \leq m \leq n-1$, by definition
\[
R_{p_j,e_m,1}(u) =  \sum_{r=1}^n \Big( [\partial_{e_n+e_{r}} p_j](\tu) u_m u_{r} - [\partial_{e_m+e_{r}} p_j](\tu) u_n u_{r} \Big).
\]
In particular,
\[
\begin{split}
B_{2,e_m}(u) &= R_{p_2,e_m,1}(u) =  
\sum_{r=1}^n \Big( [\partial_{e_n+e_{r}} p_2](\tu) u_m u_{r} - [\partial_{e_m+e_{r}} p_2](\tu) u_n u_{r} \Big) \\
&= \sum_{\substack{1 \leq r \leq n \\ r \ne n}} 2 c_{r,n} u_m u_{r} -  \sum_{\substack{1 \leq r \leq n \\ r \ne m}} 2 c_{r,m} u_n u_{r} + 2 (c_{n,n} - c_{m,m}) u_m u_n.
\end{split}
\]
If $p_j = Q(u)^{j/2}$ (for $j$ even), then by the chain and product rules,  
\begin{align*}
\quad R_{p_j,e_m,1}(u)
&= 
\sum_{r = 1}^n \Big( j (j-2) Q(\tu)^{\frac{j}{2}-2} u_{r} u_n u_m u_{r} - j (j-2) Q(\tu)^{\frac{j}{2}-2} u_m u_{r} u_n u_{r} \Big) \\
 &\quad \quad + j Q(\tu)^{\frac{j}{2}-1} \theta_n u_m u_n - j Q(\tu)^{\frac{j}{2}-1} \theta_m u_n u_m  \\
&= j Q(\tu)^{\frac{j}{2}-1} (\theta_n - \theta_m) u_m u_n = \frac{j}{2} Q(\tu)^{\frac{j}{2}-1} R_{Q,e_m,1}(u).
\end{align*}
Next we compute terms involving $|\ga|=2$, which take the form $2e_\ell$ or $e_\ell + e_{\ell'}$ for $\ell \neq \ell'$. For $1 \leq \ell \leq n-1$,
\[
\begin{split}
B_{2,2e_{\ell}}(u) 
&= R_{p_2,2e_{\ell},0}(u)
= \sum_{a=0}^2 \frac{(-1)^{a}}{a! (2-a)!} [\partial_{a e_{\ell} + (2-a) e_n} p_2](\tu) u_{\ell}^{2-a} u_n^a \\
&= c_{n,n} u_{\ell}^2 - 2 c_{{\ell},n} u_{\ell} u_n + c_{{\ell},{\ell}} u_n^2
\end{split}
\]
and for $1 \leq {\ell} \ne \ell' \leq n-1$,
\[
\begin{split}
B_{2,e_{\ell}+e_{\ell'}}(u) 
&= R_{p_2,e_{\ell}+e_{\ell'},0}(u)\\
&= \sum_{0 \leq a_1 \leq 1} \sum_{0 \leq a_2 \leq 1} (-1)^{a_1+a_2}[\partial_{a_1 e_{\ell} + a_2 e_{\ell'} + (2-a_1-a_2) e_n} p_2](\tu) u_{\ell}^{1-a_1} u_{\ell'}^{1-a_2} u_n^{a_1 + a_2} \\
&= 2(c_{n,n} u_{\ell} u_{\ell'} - c_{{\ell},n} u_{\ell'} u_n - c_{\ell',n} u_{\ell} u_n + c_{{\ell},\ell'} u_n^2).
\end{split}
\]
Similarly, if $p_j = Q(u)^{j/2}$, for $1 \leq {\ell} \leq n-1$,
\[
\begin{split}
R_{p_j,2e_{\ell},0}(u) &:= \sum_{a=0}^2 \frac{(-1)^{a}}{a! (2-a)!} [\partial_{a e_{\ell} + (2-a) e_n} p_j](\tu) u_{\ell}^{2-a} u_n^a   \\
&= \frac{1}{2} \left( j(j-2) Q(\tu)^{\frac{j}{2}-2} u_{\ell}^2 u_n^2 + j Q(\tu)^{\frac{j}{2}-1} \theta_{\ell} u_n^2 \right) - j(j-2) Q(\tu)^{\frac{j}{2}-2} u_{\ell}^2 u_n^2 \\
&\quad + \frac{1}{2} \left( j(j-2) Q(\tu)^{\frac{j}{2}-2} u_{\ell}^2 u_n^2 +  j Q(\tu)^{\frac{j}{2}-1} \theta_n u_{\ell}^2 \right) \\
&= \frac{j}{2} Q(\tu)^{\frac{j}{2}-1} (\theta_{\ell} u_n^2 + \theta_n u_{\ell}^2) = \frac{j}{2} Q(\tu)^{\frac{j}{2}-1} R_{Q,2e_{\ell},0}(u)
\end{split}
\]
and for $1 \leq {\ell} \ne \ell' \leq n-1$,
\[
\begin{split}
R_{p_j,e_{\ell}+e_{\ell'},0}(u) 
&= \sum_{0 \leq a_1 \leq 1} \sum_{0 \leq a_2 \leq 1} (-1)^{a_1+a_2}[\partial_{a_1 e_{\ell} + a_2 e_{\ell'} + (2-a_1-a_2) e_n} p_j](\tu) u_{\ell}^{1-a_1} u_{\ell'}^{1-a_2} u_n^{a_1 + a_2} \\
&= j(j-2) Q(\tu)^{\frac{j}{2}-2} u_n^2 u_{\ell} u_{\ell'} + j Q(\tu)^{\frac{j}{2}-1} \theta_n u_{\ell} u_{\ell'}\\
&\qquad - j(j-2) Q(\tu)^{\frac{j}{2}-2} u_{\ell} u_n u_{\ell'} u_n - j(j-2) Q(\tu)^{\frac{j}{2}-2} u_{\ell'} u_n u_{{\ell}} u_n \\
&\qquad + j(j-2) Q(\tu)^{\frac{j}{2}-2} u_{\ell} u_{\ell'} u_n^2 \\
&= j Q(\tu)^{\frac{j}{2}-1} \theta_n u_{\ell} u_{\ell'} = \frac{j}{2} Q(\tu)^{\frac{j}{2}-1} R_{Q,e_{\ell}+e_{\ell'},0}(u).
\end{split}
\]
Inserting these expansions into \eqref{eq:lem3} confirms that it factors into the expression claimed in (\ref{dep_only_2}).

\subsection{Proof of Lemma \ref{lem:nonvanishing2}}
Now all that remains is to prove Lemma \ref{lem:nonvanishing2}.
Suppose $p_2$ has coefficients notated as in \eqref{eq:p2def}, an index $1 \leq m \leq n-1$ is fixed, and suppose that  \eqref{eq:nonvanishing} holds, namely
\begin{multline}\label{eq:vanish2}
[ \sum_{\ell = 1}^{n-1} (c_{n,n} u_{\ell}^2 - 2 c_{{\ell},n} u_{\ell} u_n + c_{{\ell},{\ell}} u_n^2)^2   
\\
  +  \sum_{1 \leq \ell < \ell' \leq n-1}  [2 (c_{n,n} u_{\ell} u_{\ell'} - c_{{\ell},n} u_{\ell'} u_n - c_{\ell',n} u_{\ell} u_n + c_{{\ell},\ell'} u_n^2)]^2 ] (\theta_n-\theta_m) u_m u_n 
  \end{multline}
is identically equal to 
\begin{multline}\label{eq:vanish2right}
[ \sum_{\ell = 1}^{n-1} (c_{n,n} u_{\ell}^2 - 2 c_{{\ell},n} u_{\ell} u_n + c_{{\ell},{\ell}} u_n^2)  (\theta_{\ell} u_n^2  + \theta_n u_{\ell}^2)  
 \\
   + \sum_{1 \leq \ell < \ell' \leq n-1} 2 (c_{n,n} u_{\ell} u_{\ell'} - c_{{\ell},n} u_{\ell'} u_n - c_{\ell',n} u_{\ell} u_n + c_{{\ell},\ell'} u_n^2)(2\theta_n u_{\ell} u_{\ell'})  ]   \\
  \cdot \Big( \sum_{\substack{1 \leq r \leq n \\ r \ne n}}  c_{r,n} u_m u_{r} -  \sum_{\substack{1 \leq r \leq n \\ r \ne m}}  c_{r,m} u_n u_{r} +  (c_{n,n} - c_{m,m}) u_m u_n \Big). 
\end{multline}
By comparing the coefficients of $u_m^6$ and $u_m^5 u_n$ in these two identical expressions, we obtain
\[
\begin{split}
0 &= c_{n,n} \theta_n c_{m,n} \\
c_{n,n}^2 (\theta_n-\theta_m) &= -2 c_{m,n} \theta_n c_{m,n} + c_{n,n} \theta_n (c_{n,n} - c_{m,m}).
\end{split}
\]
These relations are equivalent to 
\begin{align}
c_{m,n} c_{n,n} &= 0 \label{eq:cmn-1}\\
2 \theta_n c_{m,n}^2 &= c_{n,n} (\theta_m c_{n,n} - \theta_n c_{m,m}). \label{eq:cnn-1}
\end{align}
From this we derive
\begin{align}
c_{m,n} &= 0. \label{eq:cmn0}
\end{align}
(In fact, from \eqref{eq:cmn-1} we must have $c_{m,n} = 0$ or $c_{n,n} = 0$. In the latter case, the equation \eqref{eq:cnn-1} implies $c_{m,n} = 0$ as well, so we have \eqref{eq:cmn0}.)

As a consequence, if \eqref{eq:vanish2} and \eqref{eq:vanish2right} are identical for every $1 \leq m \leq n-1$, then   $c_{m,n} = 0$ for every $1 \leq m \leq n-1$. Applying this in \eqref{eq:vanish2} and \eqref{eq:vanish2right}, and again assuming they are identical expressions, we obtain that for every $1 \leq m \leq n-1$,
\begin{multline} \label{eq:vanish3}
\left[ \sum_{\ell = 1}^{n-1} (c_{n,n} u_{\ell}^2 + c_{{\ell},{\ell}} u_n^2)^2 
+ \sum_{1 \leq \ell < \ell' \leq n-1} 4 (c_{n,n} u_{\ell} u_{\ell'} + c_{{\ell},\ell'} u_n^2)^2 \right] (\theta_n-\theta_m) u_m u_n \\ 
\equiv  \left[ \sum_{\ell = 1}^{n-1} (c_{n,n} u_{\ell}^2 + c_{{\ell},{\ell}} u_n^2)  (\theta_{\ell} u_n^2  + \theta_n u_{\ell}^2)
+ \sum_{1 \leq \ell < \ell' \leq n-1} 2 (c_{n,n} u_{\ell} u_{\ell'} + c_{{\ell},\ell'} u_n^2)(2\theta_n u_{\ell} u_{\ell'}) \right] \\
\cdot \Big( -  \sum_{\substack{1 \leq r \leq n-1 \\ r \ne m}}  c_{r,m} u_n u_{r} +  (c_{n,n} - c_{m,m}) u_m u_n \Big).
\end{multline}
 We also have, from \eqref{eq:cnn-1}, that $c_{n,n} (\theta_m c_{n,n} - \theta_n c_{m,m}) = 0$ for all $1 \leq m \leq n-1$. Either $c_{n,n} = 0$ or the second factor is zero; in the former case, then by considering the coefficient of $u_m^3 u_n^3$ in \eqref{eq:vanish3}, we obtain
$
0 = c_{m,m} \theta_n (-c_{m,m})
$
which implies $c_{m,m} = 0$ for all $1 \leq m \leq n-1$. So either way, we must have
\begin{equation} \label{eq:cnncmm2}
\theta_m c_{n,n} = \theta_n c_{m,m}
\end{equation}
for all $1 \leq m \leq n-1$.

To proceed further, \eqref{eq:cnncmm2} implies 
\begin{align*}
    \theta_n^2(c_{n,n} u_m^2 + c_{m,m} u_n^2)^2 & = (\theta_n c_{n,n} u_m^2 + \theta_n c_{m,m} u_n^2)^2 \\
  &   = (\theta_n c_{n,n} u_m^2 + \theta_m c_{n,n} u_n^2)^2 = c_{n,n}^2 (\theta_n u_m^2 + \theta_m u_n^2)^2.
    \end{align*}
    Similarly, (\ref{eq:cnncmm2}) implies
\[
(c_{n,n} u_m^2 + c_{m,m} u_n^2) (c_{n,n} - c_{m,m}) = c_{n,n}^2 (\theta_n u_m^2 + \theta_m u_n^2) (\theta_n - \theta_m).
\]
We now set $u_1, \dots, u_{n-1}$ except $u_m$ to zero in \eqref{eq:vanish3}, and plug these two relations into (\ref{eq:vanish3}). We obtain
\[
\sum_{1 \leq \ell < \ell' \leq n-1} 4 c_{\ell,\ell'}^2 u_n^4 (\theta_n-\theta_m) u_m u_n \equiv 0;
\]
repeating this process extracts this relation for each $1 \leq m \leq n-1$. If 
\begin{equation} \label{eq:extra}
\theta_n \ne \theta_m \quad \text{for some $1 \leq m \leq n-1$},
\end{equation} 
then the above implies 
\begin{equation} \label{eq:cll'}
c_{\ell,\ell'} = 0 \quad \text{for all $1 \leq \ell < \ell' \leq n-1$}.
\end{equation}
In this case, from \eqref{eq:cmn0}, \eqref{eq:cnncmm2} and \eqref{eq:cll'}, $p_2(u)$ is a multiple of $Q(u)$. 
Lemma \ref{lem:nonvanishing2} has now been established under the additional hypothesis \eqref{eq:extra}, namely that $Q(y) = \sum_i \theta_i y_i^2 \neq \pm |y|^2$. 

The remaining case is when   $\theta_1 = \dots = \theta_n$ (that is to say, $Q(y) = \pm |y|^2$). This can also be dealt with, as follows.   If $\theta_1 = \dots = \theta_n$, then \eqref{eq:vanish3} becomes the statement that for all $1 \leq m \leq n-1$,
\begin{multline} \label{eq:20}
0 \equiv  \left[ \sum_{\ell = 1}^{n-1} (c_{n,n} u_{\ell}^2 + c_{{\ell},{\ell}} u_n^2)  ( u_n^2  +  u_{\ell}^2)
+ \sum_{1 \leq \ell < \ell' \leq n-1} 2 (c_{n,n} u_{\ell} u_{\ell'} + c_{{\ell},\ell'} u_n^2)(2 u_{\ell} u_{\ell'}) \right] \\
\cdot \Big( -  \sum_{\substack{1 \leq r \leq n-1 \\ r \ne m}}  c_{r,m} u_n u_{r} +  (c_{n,n} - c_{m,m}) u_m u_n \Big).
\end{multline}
Here we consider two cases. 
 Suppose there exists an index $m$ for which \eqref{eq:20} holds because the first factor on the right-hand side is identically zero.
 The coefficient of $u_\ell^4$ (which must necessarily vanish) shows that $c_{n,n}=0$, while the coefficient of $u_n^2u_\ell^2$ shows that $c_{n,n}+c_{\ell,\ell}=0$ and hence $c_{\ell,\ell}=0$ for each $1 \leq \ell \leq n-1$. Finally the coefficient of $u_n^2u_\ell u_{\ell'}$ shows that $c_{\ell,\ell'}=0$ for all $1 \leq \ell < \ell' \leq n-1$. Together with \eqref{eq:cmn0}, this shows that $p_2 \equiv 0$, which is a multiple of $Q(u)$ (and is eliminated by a hypothesis of the lemma). 

The last case to consider is that for each $1 \leq m \leq n-1$, the second factor in \eqref{eq:20} vanishes identically. The coefficient of $u_mu_n$ (which must necessarily vanish) implies $c_{m,m} = c_{n,n}$ for all $1 \leq m \leq n-1$, and the coefficient of $u_nu_r$ shows $c_{r,m} = 0$ for all $1 \leq m \leq n-1$, $1 \leq r \leq n-1$ with $r \ne m$. Together with \eqref{eq:cmn0}, this implies $p_2(u)$ is a multiple of $|u|^2 = \pm  Q(u)$. This completes the proof of Lemma \ref{lem:nonvanishing2}.
The proof of Theorem \ref{thm_main_K} is now complete, and consequently Theorem \ref{thm_main_R} is proved.


\section*{Acknowledgements}
Anderson has been partially supported by NSF CAREER DMS-2237937, DMS-2231990, DMS-1502464, and an NSF Graduate Research Fellowship.  She thanks Andreas Seeger for helpful conversations related to this project. Maldague is supported by the National Science Foundation under Award No. 2103249.
Pierce has been partially supported by NSF  CAREER grant DMS-1652173, DMS-2200470, a Sloan Research Fellowship, a Joan and Joseph Birman Fellowship, a Simons Fellowship, and a Guggenheim Fellowship  during portions of this work, and thanks the Hausdorff Center for Mathematics for productive visits as a Bonn Research Chair. Yung is partially supported by a Future Fellowship FT200100399 from the Australian Research Council.

\bibliographystyle{alpha}
\bibliography{AnalysisBibliography}
\label{endofproposal}

\end{document}